\documentclass{article}
\usepackage{fullpage}
\usepackage{authblk}
\usepackage[hidelinks]{hyperref}

\usepackage[utf8]{inputenc} 
\usepackage[T1]{fontenc}   
\usepackage{lmodern}

\usepackage{amsmath}
\usepackage{amsfonts}       
\usepackage{amssymb}
\usepackage{amsthm}
\usepackage{mathtools}

\usepackage[round]{natbib}

\usepackage{pdfpages}

\usepackage{xcolor}
\usepackage{enumerate}
\usepackage{csquotes}

\usepackage[linesnumbered,ruled,vlined]{algorithm2e}

\usepackage[linesnumbered,ruled,vlined]{algorithm2e}
\usepackage{listings}
\usepackage[format=plain,labelfont=bf]{caption}

\newcommand\blfootnote[1]{%
  \begingroup
  \renewcommand\thefootnote{}\footnote{#1}%
  \addtocounter{footnote}{-1}%
  \endgroup
}

\newcommand{\red}{\textcolor{black}}

\newcommand{\cB}{{\mathcal B}}

\newcommand{\cP}{{\mathcal P}}
\newcommand{\cT}{\mathcal T_1}
\newcommand{\cD}{{\mathcal D}}
\newcommand{\0}{\mathit 0_\mathcal B}

\theoremstyle{definition}
\newtheorem{definition}{Definition}
\newtheorem{example}{Example}
\newtheorem{remark}{Remark}

\theoremstyle{plain}
\newtheorem{lemma}{Lemma}
\newtheorem{theorem}{Theorem}
\newtheorem{proposition}{Proposition}
\newtheorem{corollary}{Corollary}

\title{Combinatorial perspectives on Dollo-$k$ characters in phylogenetics}

\author[1]{Remco Bouckaert}
\author[2,$\ast$]{Mareike Fischer}
\author[3]{Kristina Wicke}

\affil[1]{School of Computer Science, The University of Auckland, Auckland, New Zealand}
\affil[2]{Institute of Mathematics and Computer Science, University of Greifswald, Greifswald, Germany}
\affil[3]{Department of Mathematics and Mathematical Biosciences Institute, The Ohio State University, Columbus OH, USA}
\date{}

\begin{document}
\maketitle

\begin{abstract} 
Recently, the perfect phylogeny model with persistent characters has attracted great attention in the literature. It is based on the assumption that complex traits or characters can only be gained once and lost once in the course of evolution. Here, we consider a generalization of this model, namely Dollo parsimony, that allows for multiple character losses. More precisely, we take a combinatorial perspective on the notion of Dollo-$k$ characters, i.e. traits that are gained at most once and lost precisely $k$ times throughout evolution. We first introduce an algorithm based on the notion of spanning subtrees for finding a Dollo-$k$ labeling for a given character and a given tree in linear time. We then compare persistent characters (consisting of the union of Dollo-0 and Dollo-1 characters) and general Dollo-$k$ characters. While it is known that there is a strong connection between Fitch parsimony and persistent characters, we show that Dollo parsimony and Fitch parsimony are in general very different. Moreover, while it is known that there is a direct relationship between the number of persistent characters and the Sackin index of a tree, a popular index of tree balance, we show that this relationship does not generalize to Dollo-$k$ characters. In fact, determining the number of Dollo-$k$ characters for a given tree is much more involved than counting persistent characters, and we end this manuscript by introducing a recursive approach for the former. This approach leads to a polynomial time algorithm for counting the number of Dollo-$k$ characters, and both this algorithm as well as the algorithm for computing Dollo-$k$ labelings are publicly available in the Babel package for BEAST 2.
\end{abstract}


\blfootnote{$^\ast$Corresponding author\\ \textit{Email address:} \texttt{email@mareikefischer.de} (Mareike Fischer)}

\section{Introduction}
Based on Dollo's law of irreversibility \citep{Dollo1893}, also referred to as Dollo's law or Dollo's principle, \emph{Dollo parsimony} \citep{Farris1977} is a well-known model in character-based phylogeny reconstruction. It is based on the assumption that a trait that has been lost throughout the course of evolution cannot be regained (i.e. there is no parallel evolution). More precisely, Dollo parsimony assumes that a trait  may only be gained once but may be lost multiple times. It is less restrictive than the so-called \emph{perfect phylogeny model with persistent characters} \citep{Bonizzoni2012} that has recently attracted great attention in the literature \citep{Bonizzoni2012} (see also \citet{Bonizzoni2014,Bonizzoni2016,Bonizzoni2017,Wicke2018}). In this model, a trait can both be gained and lost at most once. Dollo parsimony and variants of it have thus been proven useful in various fields, e.g. in the analysis of intron conservation patterns and the evolution of alternatively spliced exons \citep{Alekseyenko2008}, the evolution of tumor cells in cancer phylogenetics (see, e.g., \citet{Bonizzoni2017a, Bonizzoni2019, Ciccolella2018, El-Kebir2018}), or the evolution of multidomain proteins \citep{Przytycka2006}, but also in the evolution of languages (see, e.g., \citet{Nicholls2008,Bouckaert2017}).

While the general Dollo parsimony model allows for multiple character losses, the Dollo-$k$ parsimony model restricts the number of losses to at most $k$.  For $k=1$, this model thus corresponds to the perfect phylogeny model with persistent characters \citep{Bonizzoni2012}. 

Here, we focus on the notion of binary Dollo-$k$ characters, i.e. binary characters (taking the values 0 and 1, usually interpreted as the presence or absence of some complex trait in a species) that can be realized on a rooted phylogenetic tree $T$ by at most one gain and exactly $k$ losses along the edges of $T$. Note that in contrast to the Dollo-$k$ parsimony model we assume \emph{exactly} $k$ losses instead of \emph{at most} $k$ losses in order to obtain concise mathematical results and characterizations. However, our results for `exactly' $k$ can easily be generalized to `at most' $k$ by considering the union of Dollo-0, Dollo-1, $\ldots$, Dollo-$(k-1)$, and Dollo-$k$ characters.

The main aim of this paper is to provide a new combinatorial perspective on Dollo-$k$ characters. In particular, we link and contrast general Dollo-$k$ characters with the family of persistent characters (i.e. binary characters that can be realized on a rooted phylogenetic tree $T$ by at most one gain and one loss along the edges of $T$), for which a thorough combinatorial analysis was recently performed \citep{Wicke2018}.

We begin by describing a linear time algorithm for finding a Dollo-$k$ labeling for a given binary character and a given tree (i.e. a labeling of the internal nodes of the tree such that there is at most one $0 \rightarrow 1$ transition followed by exactly $k$ $1 \rightarrow 0$ transitions) based on the calculation of a certain spanning tree, which we call the `1-tree'. While this algorithm can also be used to decide whether a given character is persistent, we then show that Dollo-$k$ characters and persistent characters are in general very different. First, while it was shown in \citep{Wicke2018} that there is a close relationship between the so-called Fitch algorithm \citep{Fitch1971} and persistent characters, Dollo-$k$ parsimony and Fitch parsimony are in general very different concepts. Second, while \citep{Wicke2018} also showed that there is a one-to-one correspondence between the balance of a tree (in terms of its Sackin index \citep{Sackin1972}) and the number of persistent characters it induces, this correspondence does not hold for general Dollo-$k$ characters. In fact, counting the number of Dollo-$k$ characters for a given tree and fixed $k$ is much more involved than counting the number of persistent characters,  but we provide a recursive approach for it. We end by discussing our results and indicating some directions for future research.

\section{Preliminaries}
Before we can present our results, we need to introduce some definitions and notations.

\subsection*{Phylogenetic trees and related concepts}
Throughout this manuscript, let $X$ denote a finite set (e.g., of species or taxa) of size $n$, where we assume without loss of generality that $X=\{1, \ldots, n\}$.
\red{\begin{definition}[Rooted binary phylogenetic $X$-tree]
A \emph{rooted binary phylogenetic $X$-tree $T$} is a rooted tree $T=(V(T),E(T))$ (or, whenever there is no ambiguity,  $T=(V,E)$ for short) with root $\rho$, root edge $(\rho', \rho)$, node set $V(T)$ (or $V$), and edge set $E(T)$ (or $E$), whose leaves except for $\rho'$ are identified with $X$ (i.e. $T$ is a \emph{leaf-labeled} tree), and where each node $v \in V \setminus (X \cup \{\rho'\})$ has degree 3. 
\end{definition}}
Note that even though $\rho'$ is mathematically a leaf (i.e. a degree-1 vertex), it is not identified with one of the taxa in $X$ but is considered as an ancestral species instead. Biologically, while $\rho$ can be interpreted as the most recent common ancestor of the taxa in $X$, $\rho'$ can be considered as some past ancestor of $\rho$ itself (Figure \ref{Fig_Tree}).\footnote{While for most phylogenetic studies, considering $\rho$ suffices, concerning the Dollo model we need to take $\rho'$ into account, too, for technical reasons.} Moreover, we refer to the set of degree-3 nodes of $T$ as \emph{internal nodes} and denote it by $\mathring{V}$ (i.e. $\mathring{V} = V \setminus (X \cup \{\rho'\})$).
If $n=1$, $T$ consists only of the root edge $(\rho',\rho)$. For technical reasons, $\rho$ is in this case at the same time defined to be the root and the only leaf in $T$.

We implicitly assume that all edges in $T$ are directed away from $\rho'$. Given two nodes $u, v \in V$, we call $u$ an \emph{ancestor} of $v$ and $v$ a \emph{descendant} of $u$ if there exists a directed path from $u$ to $v$. Note that for technical reasons, we also assume a node $v$ to be its own ancestor and descendant. Moreover, if $u$ and $v$ are connected by an edge $(u,v)$, we say that $u$ is the \emph{direct ancestor} or \emph{parent} of $v$, and $v$ is the \emph{direct descendant} or \emph{child} of $u$. Furthermore, two nodes that have the same parent are called \emph{siblings}. Two leaves, say $x$ and $y$, that have the same parent are also called a \emph{cherry}, and we denote this by $[x,y]$. 
Finally, given a set $\widetilde{V} = \{v_1, v_2, \ldots \} \subseteq V$ of nodes of $T$, we call the last node that lies on all directed paths from $\rho'$ to an element of $\widetilde{V}$, the \emph{most recent common ancestor} (MRCA) of the nodes in $\widetilde{V}$.

Furthermore, given a node $u$ of $T$ (which may be a leaf), we denote the subtree of $T$ rooted at $u$ by $T_u$, and use $v_u$ and $n_u$ to denote the number of nodes, respectively leaves, of $T_u$. Moreover, given a rooted binary phylogenetic tree $T$ on $n \geq 2$ leaves, we often consider the \emph{standard decomposition} of $T$ into its two maximal pending subtrees rooted at the children, say $a$ and $b$, of the root, and denote this by $T=(T_a, T_b)$ (note that in this case, edges $(\rho, a)$ and $(\rho,b)$ are the root edges of $T_a$ and $T_b$, respectively).

\red{Three particular types of phylogenetic trees appearing in this manuscript are the caterpillar, semi-caterpillar, and fully balanaced tree.}
\red{\begin{definition}[(Semi-)caterpillar tree]
A rooted binary phylogenetic tree that has precisely one cherry is called a \emph{(rooted) caterpillar tree}. For technical reasons, if $n=1$ we also consider the unique rooted binary phylogenetic tree with one leaf as a rooted caterpillar tree.
Moreover, a tree $T=(T_a, T_b)$ such that both $T_a$ and $T_b$ are rooted caterpillar trees is called a \emph{semi-caterpillar tree}. 
\end{definition}}

Note that every caterpillar tree is also a semi-caterpillar tree, whereas the converse holds only for the case where  either $T_a$ or $T_b$ (or both if $n=2$) have precisely one leaf.
\red{\begin{definition}[Fully balanced tree of height $h$]
A rooted binary phylogenetic tree with $n=2^h$ leaves with $h \in \mathbb{N}_{\geq 0}$ is called a \emph{fully balanced tree of height} $h$, denoted by $T_h^\mathit{fb}$, if and only if all of its leaves have depth exactly $h$, where the depth of a node $v$ is the number of edges on the path from $\rho$ to $v$.
\end{definition}}

Note that for $n \geq 2$, both maximal pending subtrees of a fully balanced tree are again fully balanced trees, and we have $T_h^{\mathit{fb}}=(T^{\mathit{fb}}_{h-1}, T^{\mathit{fb}}_{h-1})$.

\subsection*{Dollo-$k$ characters and labelings}
Having introduced the concept of a phylogenetic $X$-tree $T$, we now turn to the data that we will map onto the leaves of $T$. \red{We begin by introducing the concept of binary characters.}
\red{\begin{definition}[Binary character]
A \emph{(binary) character} $f$ is a function $f: X \rightarrow \{0,1\}$ from the leaf set $X$ to the state set $\{0,1\}$.
\end{definition}}

We often abbreviate a character  $f$ by denoting it as $f = f(1)f(2) \ldots f(n)$. As an example, for the rooted phylogenetic tree on 6 leaves depicted in Figure \ref{Fig_Tree}(a), character $f=100001$ assigns state 1 to leaves 1 and 6, and state 0 to leaves 2, 3, 4 and 5. 
Moreover, given a binary character $f$ we use $\bar{f}$ to denote its inverted character (where we replace all zeros by ones and vice versa), e.g. for $f=100001$, we have $\bar{f}=011110$.

\red{Further important concepts related to characters are extensions and the changing number.}
\red{\begin{definition}[Extension, change edge, and changing number]
An \emph{extension} of a binary character $f$ is a map $g: V \rightarrow \{0,1\}$ such that $g(x)=f(x)$ for all $x \in X$. 
An edge $e=(u,v) \in E$ with $g(u) \neq g(v)$ is called a \emph{change edge} of $g$ on $T$, and $ch(g) \coloneqq |\{e \in E: \mbox{ $e$ is a change edge}\}|$ denotes the \emph{changing number} of $g$ on $T$.
When $g(u)=0$ and $g(v)=1$ we often loosely refer to $e=(u,v)$ as a $0 \rightarrow 1$ transition, and analogously, if $g(u)=1$ and $g(v)=0$, we speak of a $1 \rightarrow 0$ transition.
\end{definition}}

Based on this we can now turn to Dollo-$k$ characters. Throughout this manuscript, given a rooted phylogenetic tree $T$ with root edge $(\rho',\rho)$, we assume that the state of $\rho'$ is 0. 
\red{\begin{definition}[Dollo-$k(f,T)$ character and Dollo score]
A binary character $f$ is called a \emph{Dollo-$k(f,T)$ character} if it can be realized on $T$ by at most one $0 \rightarrow 1$ transition (representing a `gain' of some trait) followed by exactly $k(f,T)$ $1 \rightarrow 0$ transitions (representing $k$ `losses' of this trait), and such that $k(f,T)$ is minimal \emph{minimal} in the sense that $f$ cannot be realized by fewer $1 \rightarrow 0$ transitions. Whenever there is no ambiguity concerning $T$, we just refer to $f$ as a Dollo-$k$ character. More formally, $f$ is called a Dollo-$k$ character if there exists an extension $g$ of $f$ that realizes $f$ with at most one $0 \rightarrow 1$ transition followed by exactly $k$ $1 \rightarrow 0$ transitions and such that $k$ is minimal. $k(f,T)$, i.e. the minimal number of $1 \rightarrow 0$ transitions required to realize $f$ on $T$, is also often referred to as the \emph{Dollo score of $f$ on $T$}.
\end{definition}}

Note that only the constant character $f=0 \ldots 0$ can be realized on $T$ without a $0 \rightarrow 1$ transition (in fact, it can be realized without any transition at all). In all other cases, there must be precisely one $0 \rightarrow 1$ transition preceding the $k$ $1 \rightarrow 0$ transitions. 

\red{\begin{definition}[Dollo-$k$ labeling]
Given a rooted binary phylogenetic $X$-tree $T$ and a binary character $f$ on $X$, a minimal extension that realizes $f$ (minimal in the sense that it minimizes the number of $1 \rightarrow 0$ transitions) is called a \emph{Dollo-$k$ labeling} for $f$ on $T$.
\end{definition}
We will show in Theorem \ref{thm_dollo_labeling} that this Dollo-$k$ labeling is unique.} Moreover, we often call the $0 \rightarrow 1$ transition a \emph{birth event} and refer to the $1 \rightarrow 0$ transitions as \emph{death events} or \emph{losses}. 
As an example, consider tree $T$ and character $f$ depicted in Figure \ref{Fig_Tree}(a). $f$ can be realized by a birth event on edge $(\rho', \rho)$ and three death events (on the edges incident to leaves 2, 3, and 4, respectively). In particular, $f$ cannot be realized by fewer than three $1 \rightarrow 0$ transitions, and thus $f$ is a Dollo-$3$ character on $T$. In other words, $k(f,T)=3$.\\

Finally, note that the union of Dollo-0 and Dollo-1 characters corresponds to the set of so-called \emph{persistent} characters, i.e. binary characters that can be realized on a rooted phylogenetic tree $T$ by at most one $0 \rightarrow 1$ transition followed by at most one $1 \rightarrow 0$ transition \citep{Bonizzoni2012, Wicke2018}. Our results for general Dollo-$k$ characters will thus also shed new light on persistent characters.

\subsection*{The Fitch algorithm}
In Section \ref{Sec_Fitch} we will relate Dollo-$k$ labelings to the well-known \emph{Fitch algorithm} \citep{Fitch1971} that can be used to calculate the so-called \emph{parsimony score} of a character $f$ on a rooted binary phylogenetic tree $T$ with root edge $(\rho',\rho)$, where $g(\rho')=0$. 
\red{\begin{definition}[Parsimony score and most parsimonious extension]
The parsimony score $l(f,T)$ is defined as 
$$l(f,T) \coloneqq \begin{cases}
\min_g ch(g,T) &\quad \text{if } g(\rho)=0, \\
\min_g ch(g,T) -1 &\quad \text{if } g(\rho)=1,
\end{cases}$$
where the minimum is taken over all possible extensions $g$ of $f$ but a change on the root edge $(\rho',\rho)$ does not increase the score.\footnote{Note that the usual definition of $l(f,T)$ as found in the literature (which states $\min_g ch(g,T)$ regardless of $g(\rho)$) only refers to trees without a root edge. On such trees, our definition coincides with this notion. The second case here is only needed to account for a possible change on the root edge.} An extension $g$ that minimizes the changing number is called a \emph{most parsimonious} extension.
\end{definition}}

In case of rooted binary phylogenetic trees, both the parsimony score as well as a most parsimonious extension can be computed with the Fitch algorithm, which we will now introduce. Formally, this algorithm consists of multiple phases, but here we will only consider the first two phases. 

The first phase is based on \emph{Fitch's parsimony operation} $\ast$, which is defined as follows. Let $\mathcal{R}$ be a set of states (in our case $\mathcal{R}=\{0,1\}$), and let $A, B \subseteq \mathcal{R}$. Then,
$$ A \ast B \coloneqq \begin{cases}
A \cap B &\text{ if } A \cap B \neq \emptyset, \\
A \cup B &\text{ otherwise.} 
\end{cases}$$
In principle, the first phase (a `bottom-up' phase) of the Fitch algorithm traverses $T$ from the leaves to $\rho$ and assigns each parental node a state set (also called `ancestral state set') based on the state sets of its children. First, each leaf is assigned the set consisting of the state assigned to it by $f$ and then all other nodes of $T$ (except for $\rho'$), whose children have already been assigned state sets $A$ and $B$, are assigned set $A \ast B$. Then, the parsimony score $l(f,T)$ corresponds to the number of times the union is taken \citep{Fitch1971}. 
Note, however, that the Fitch algorithm does usually not consider trees with root edges and a potential change on the root edge $(\rho',\rho)$ (recall that we always assume that $\rho'$ is in state 0) is not taken into account when calculating the parsimony score of $f$ on $T$.
Moreover, note that given a binary character $f$ and its inverted version $\bar{f}$, we have $l(f,T)=l(\bar{f},T)$ as Fitch's parsimony operation is a set operation and the roles of zeros and ones are interchangeable.

The second phase of the Fitch algorithm (a `top-down' phase) then traverses $T$ from the root to the leaves to compute a most parsimonious extension. First, the root $\rho$ is arbitrarily assigned one state of its state set that was computed during the first phase of the algorithm. Then, for every internal node $v$ that is a child of a node $u$ that has already been assigned a state, say $g(u)$, we set $g(v)=g(u)$ if $g(u)$ is contained in the state set of $v$. Otherwise, we arbitrarily assign any state from the state set of $v$ to $v$. As the word `arbitrarily' suggests, the most parsimonious extension is not necessarily unique. Moreover, there might be most parsimonious extensions that cannot be found by the second phase of the Fitch algorithm as described here (as we omitted the correction phase that adds more states to the state sets reconstructed by the first phase when appropriate \citep{Fitch1971, Felsenstein2004}), but this is not relevant for our purposes.

As an example, consider tree $T$ and character $f$ depicted in Figure \ref{Fig_Tree}(a). Part (b) of Figure \ref{Fig_Tree} depicts the state sets assigned to the nodes of $T$ (except for $\rho'$) by the first phase of the Fitch algorithm. The $\{0,1\}$ sets assigned to the parents of leaves 1 and 2, respectively leaves 5 and 6, correspond to a union being taken, and thus the parsimony score of $f$ on $T$ equals 2. It is easily verified that applying the second phase of the Fitch algorithm then results in a unique most parsimonious extension for $f$ on $T$ that assigns state 0 to all internal nodes. Note that this unique most parsimonious extension for $f$ is different from the unique Dollo-$3$ labeling for $f$ depicted in Figure \ref{Fig_Tree}(a). We will elaborate on the relationship between most parsimonious extensions and Dollo-$k$ labelings in Section \ref{Sec_Fitch}, where we show in particular that the difference between the parsimony score $l(f,T)$ and the Dollo score $k(f,T)$ can be made arbitrarily large.

\begin{figure}[htbp]
	\centering
	\includegraphics[scale=0.19]{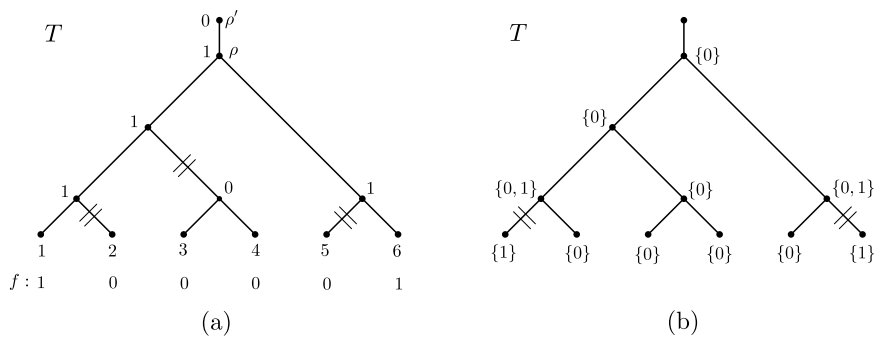}
	\caption{(a) Rooted binary phylogenetic tree $T$ with leaf set $\{1, \ldots,5\}$, root $\rho$, and root edge $(\rho', \rho)$. The binary character $f=100001$ is a Dollo-$3$ character as it requires one $0 \rightarrow 1$ transition on the root edge and three $1 \rightarrow 0$ transitions on the edges leading to leaves 2, 5, and the parent of leaves 3 and 4 (it can easily be verified that $f$ cannot be realized by fewer $1 \rightarrow 0$ transitions). The corresponding unique Dollo-$3$ labeling is given by the extension of $f$ depicted. (b) State sets resulting from applying the first phase of the Fitch algorithm for character $f=100001$. Note that the Fitch algorithm does not recover an ancestral state set for node $\rho'$. Moreover, note that the unique most parsimonious extension of $f$ would assign state $0$ to all internal nodes (yielding a parsimony score of 2; the two edges requiring a change are highlighted). Thus, the Dollo-$3$ labeling for $f$ and the most parsimonious extension of $f$ do not coincide.}
	\label{Fig_Tree}
\end{figure}

\section{Computing and characterizing Dollo-$k$ labelings}
The main aim of this section is to fully characterize Dollo-$k$ labelings and present an algorithm for finding a Dollo-$k$ labeling and determining $k$ for a given tree $T$ and binary character $f$ in linear time. 

However, we start by analyzing a basic property of Dollo-$k$ characters, namely the range of values $k$ can take for a rooted binary phylogenetic tree $T$ on $n$ leaves. 

\subsection{Basic properties of Dollo-$k$ characters}

We now state the intuitive fact that given a rooted binary phylogenetic tree $T$ on $n$ leaves and a binary character $f$, the Dollo score of $f$ on $T$ lies between 0 and $n-2$.  

\begin{proposition} \label{dollo_0_n-2}
Let $T$ be a rooted binary phylogenetic tree with $n$ leaves and let $f$ be a binary character. Then, 
$$ 0 \leq k(f,T) \leq n-2.$$
\end{proposition}

\begin{proof}
First, as $k(f,T)$ is the minimum number of losses required to realize $f$ on $T$, $k(f,T)$ is clearly non-negative. However, for every tree $T$, there are characters that can be realized by zero losses (e.g., $f=0 \ldots 0$ and $f=1\ldots 1$), and thus 0 is a sharp lower bound for $k(f,T)$.

Now, for the upper bound on $k(f,T)$, notice that we cannot have more than one loss on any of the $n$ unique paths from $\rho'$ to the elements of $X$. Thus, clearly $k(f,T) \leq n$. Now suppose that $k(f,T)=n$. This implies that all elements of $X$ are in state 0, i.e. $f=0 \ldots 0$. However, in this case clearly $k(f,T)=0$, because we require neither a gain nor a loss to realize $f$ on $T$.
Similarly, suppose that $k(f,T)=n-1$. As in the Dollo model a trait cannot be regained once it is lost, a loss edge cannot be ancestral to any other loss edge. This necessarily implies one of the following:
\begin{enumerate}
\item [(i)] there are $n-1$ losses on $n-1$ distinct edges incident to leaves (i.e., there are $n-1$ leaves in state 0 and 1 leaf in state 1).
\item [(ii)] there is 1 loss on an edge incident to a cherry and $n-2$ losses on distinct edges incident to leaves (which implies that all leaves are in state 0). 
\end{enumerate}
However, in both cases, $k(f,T)=0$ (in case (i), $f$ could be realized by a birth event on the edge incident to the single leaf in state 1; in case (ii), no birth or death events are required at all). 
Thus, we can conclude that $k(f,T) \leq n-2$.
This bound is again sharp: Let $T$ be a caterpillar tree with $n$ leaves and let $f$ be such that it assigns state 1 to one leaf of the cherry of $T$ and to the leaf that is a child of the root, whereas $f$ assigns state 0 to all other leaves (see Figure \ref{Fig_DolloPars}). Then, $k(f,T)=n-2$ as $f$ requires the birth event to be placed on the root edge $(\rho',\rho)$ and induces $n-2$ losses on the edges incident to the $n-2$ leaves in state 0.
This completes the proof.
\end{proof}

\begin{figure}[htbp]
	\centering
	\includegraphics[scale=0.25]{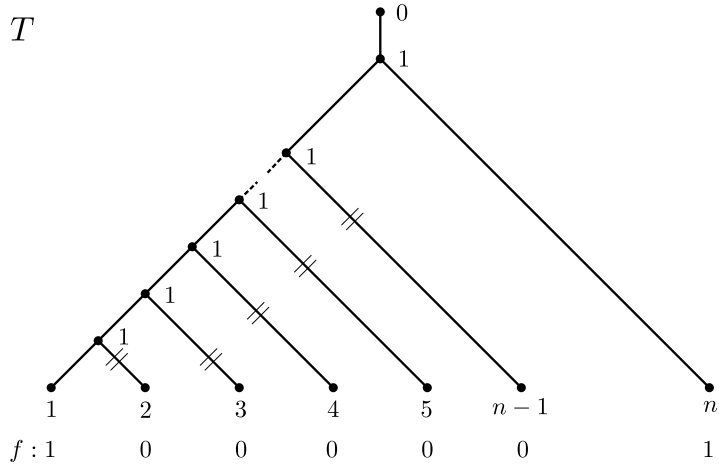}
	\caption{Rooted binary phylogenetic tree $T$ on leaf set $\{1, \ldots, n\}$ and a binary character $f$. $f$ is a Dollo-$(n-2)$-character (requiring $n-2$ death events/losses on the edges directed into the $n-2$ leaves in state $0$).}
	\label{Fig_DolloPars}
\end{figure}

\subsection{The $\cB$-node and the 1-tree}

We now introduce two concepts that will be central to the remainder of this manuscript (in particular to the linear time algorithm for finding a Dollo-$k$ labeling), namely the so-called $\cB$-node (birth node) and the 1-tree. 

\begin{definition}[$\cB$-node] \label{def_bnode}
Let $T$ be a rooted binary phylogenetic tree and let $f$ be a binary character. Then the $\cB$-node (\emph{birth node}) is defined to be the MRCA of all leaves in state 1.
\end{definition}

We now state an intuitive but crucial property of the $\cB$-node. 
\begin{lemma}\label{lemma_Bnode_birthevent}
Let $T$ be a rooted binary phylogenetic tree and let $f$ be a binary character. Consider a Dollo-$k$ labeling for $f$ on $T$. Then, there is a birth event on the edge directed into the $\cB$-node.
\end{lemma}

\begin{proof}
As the $\cB$-node is the MRCA of all leaves in state 1, the birth event cannot be on an edge below the $\cB$-node; otherwise, $f$ would require a second birth event. 
However, the birth event cannot be on an edge above the parent of the $\cB$-node in a Dollo-$k$ labeling, either. If it was, excessive death events would be required to realize $f$ on $T$ and the corresponding labeling would not be a Dollo-$k$ labeling (as $k$ would not be minimal). This completes the proof.
\end{proof}

Biologically, given a rooted phylogenetic tree $T$ and binary character $f$ representing e.g. the presence or absence of some complex trait, the $\cB$-node represents some ancestor of the present-day species where this trait \emph{first} occurred. The $\cB$-node is thus of particular relevance when studying the \emph{origin} of traits or characteristics present in some of today's living species. 
In the following we will see that the $\cB$-node is also fundamental for finding a Dollo-$k$ labeling for a given binary character $f$. 

\begin{definition}[1-tree] \label{def_1tree}
Let $T$ be a rooted binary phylogenetic tree and let $f$ be a binary character. Then, the minimum \red{induced} subtree of $T$ connecting all leaves that are assigned state 1 by $f$ is called the \emph{1-tree} of $f$ on $T$ and is denoted by $\cT(f)$.
\end{definition}

\begin{remark}\label{remark_steiner_tree}
Note that $\cT(f)$ can be considered as a `spanning subtree' of $T$ that spans all leaves of $T$ that are assigned state 1 by $f$. More precisely, it is the unique minimum \red{induced} subtree of $T$ that connects the leaves of $T$ that are in state 1. Note that finding a minimum subtree spanning a specified set of nodes (also called `terminals') in a given edge-weighted graph is formally an instance of the so-called `minimum Steiner tree problem' which is known to be NP-complete \citep{Karp1972}. However, while we can interpret the problem of computing the 1-tree as an instance of the minimum Steiner tree problem (by specifying the leaves in state 1 as the set of terminals and using e.g. unit edge weights), this instance if of course easy to solve as the solution is simply the unique \red{induced} subtree of $T$ connecting all leaves in state 1.
\end{remark}

Moreover, note that the 1-tree is a rooted phylogenetic tree \emph{without} root edge and possibly with additional degree-2 vertices. In particular, the root of the 1-tree has degree 2 whenever $\cT(f)$ has at least two leaves (if there is only one leaf of $T$ that is assigned state 1 by $f$, $\cT(f)$ consists of only one node which is at the same time considered to be the root and only leaf of $\cT(f)$). In fact, the root of the 1-tree is the $\cB$-node. Moreover, all 0's that are descending from the 1-tree in $T$ lead to degree-2 vertices in the 1-tree.

\begin{lemma}\label{lemma_1tree_root}
Let $T$ be a rooted binary phylogenetic tree and let $f$ be a binary character. Then the $\cB$-node is the root of the 1-tree $\cT(f)$.
\end{lemma}

\begin{proof}
First, note that if $f=0 \ldots 0$, there is no $\cB$-node in $T$, and the 1-tree $\cT(f)$ is empty. Now, if $f$ assigns state 1 to precisely one leaf $x \in X$, $x$ is the $\cB$-node and $\cT(f)$ consists only of $x$. In particular, the $\cB$-node is the root of $\cT(f)$. Finally, assume that $f$ assigns state 1 to at least two leaves. Then, $\cT(f)$ consists of the union of (undirected) paths in $T$ that connect any two leaves in state 1. As the $\cB$-node is the MRCA of all leaves in state 1, it is necessarily contained in $\cT(f)$ (as there are at least two leaves, say $x$ and $y$, in state 1 such that the unique (undirected) path between $x$ and $y$ visits the $\cB$-node). However, as no edge that is not descending from the $\cB$-node can be required to connect two leaves in state 1, the $\cB$-node must in fact be the root of $\cT(f)$. This completes the proof.
\end{proof}

As an example, consider tree $T$ and character $f$ depicted in Figure \ref{Fig_1tree}. The 1-tree $\cT(f)$ is the subtree of $T$ spanning leaves 1,2, and 4, and the root of $\cT(f)$ is the $\cB$-node. In particular, $f$ can be realized on $T$ as a Dollo-1 character by placing the birth event on the edge directed into the $\cB$-node and placing a single death event on the edge directed into leaf 3.

\begin{figure}[htbp]
\centering
\includegraphics[scale=0.25]{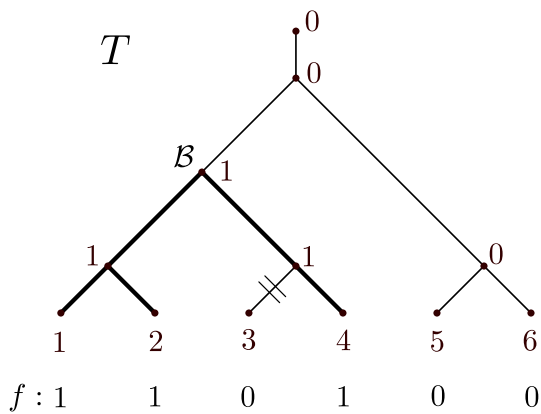}
\caption{Rooted binary phylogenetic tree $T$ and binary character $f$. The 1-tree $\cT(f)$ is depicted in bold; the root of $\cT(f)$ is the $\cB$-node. Additionally, the (unique) Dollo-1 labeling for $f$ is shown. It assigns state 1 to all nodes in $V(\cT(f))$ and state 0 to all nodes in $V(T) \setminus V(\cT(f))$.}
\label{Fig_1tree}
\end{figure}

\subsection[A linear time algorithm for finding Dollo-$k$ labelings]{A linear time algorithm for finding Dollo-$\boldsymbol{k}$ labelings}
Based on the 1-tree we can now state the first main theorem of this paper. In particular, for a given tree $T$ and character $f$, we will derive the uniqueness of the Dollo labeling, and we will show how the 1-tree can be used to find it.

\begin{theorem}\label{thm_dollo_labeling}
Let $T$ be rooted binary phylogenetic tree and let $f$ be a binary character. Then $f$ has a unique Dollo-$k$ labeling and this labeling consists of assigning state 1 to all nodes of $T$ that are in the 1-tree $\cT(f)$ and assigning state 0 to all remaining nodes (i.e. we set $g(v)=1$ for $v \in V(T) \cap V(\cT(f))=V(\cT(f))$ and $g(v)=0$ for $v \in V(T) \setminus V(\cT(f))$).
\end{theorem}

As an example, consider tree $T$ and character $f$ depicted in Figure \ref{Fig_1tree}. The 1-tree $\cT(f)$ is depicted in bold and the Dollo-$k$ labeling for $f$ is given by assigning state 1 to all nodes in $\cT(f)$ and state 0 to all other nodes.

In order to prove Theorem \ref{thm_dollo_labeling}, we require some additional lemmas that further characterize Dollo-$k$ labelings.

\begin{lemma}[Zero lemma] \label{zero_lemma}
Let $T$ be a rooted binary phylogenetic tree and let $f$ be a binary character. Then, in a Dollo-$k$ labeling for $f$, every node not descending from the $\cB$-node is assigned state 0.
\end{lemma}

\begin{proof}
Let $v$ be a node that is not descending from the $\cB$-node, and suppose that $v$ is assigned state 1 in a Dollo-$k$ labeling for $f$. As $v$ is not descending from the $\cB$-node, by Definition \ref{def_bnode}, all leaves of the subtree $T_v$ of $T$ rooted at $v$ must be in state 0. This implies that assigning state 1 to $v$ induces at least one death event that could be avoided by assigning state 0 to $v$. As a Dollo-$k$ labeling for $f$ is by definition an extension of $f$ that minimizes the number of death events required to realize $f$ on $T$, it thus cannot be the case that $v$ is assigned state 1 in a Dollo-$k$ labeling. In particular, all nodes not descending from the $\cB$-node must be assigned state 0 in a Dollo-$k$ labeling. This completes the proof.
\end{proof}

The next lemma implies that any Dollo-$k$ labeling for a character $f$ will assign state 1 to all nodes in the 1-tree.

\begin{lemma}[One lemma] \label{one_lemma}
Let $T$ be rooted binary phylogenetic tree and let $f$ be a binary character. Then, in a Dollo-$k$ labeling for $f$, every node on the path from the $\cB$-node to a leaf in state 1 is assigned state 1.
\end{lemma}

\begin{proof}
If $f=0\ldots0$, there is no $\cB$-node and the statement is obviously satisfied. 
If there is precisely one leaf in state 1, this leaf is the $\cB$-node and the statement trivially holds. 
Finally, if there is more than one leaf in state 1, assume, for the sake of a contradiction, that a path, say $P$, between the $\cB$-node and one of the leaves in state 1, say $x$, contains a node $v$ assigned state 0. By Lemma \ref{lemma_Bnode_birthevent}, $v$ cannot be the $\cB$-node (because then $v$ would be in state 1). 
However, as by Lemma \ref{lemma_Bnode_birthevent} there is a birth event on the edge directed into the $\cB$-node, $v$ is in state 0, and $x$ is in state 1, there must be a second birth event on an edge on the path from $v$ to $x$, which is a contradiction. Thus, all nodes on the path from the $\cB$-node to a leaf in state 1 must be assigned state 1. This completes the proof.
\end{proof}

In order to state the last technical lemma, we need the following definition.
\begin{definition}[$\mathit{0}$-node, $\mathit{0}$-clade, maximal $\mathit{0}$-node, maximal $\0$-node]
Let $T$ be a rooted binary phylogenetic tree and let $f$ be a binary character. Then a subtree $T_u$ of $T$ that only contains leaves in state 0 is called a \emph{$\mathit{0}$-clade} and $u$ is called a \emph{$\mathit{0}$-node}. If $T_u$ is maximal in the sense that there is no subtree $T_{u'}$ of $T$ with $V(T_u) \subset V(T_{u'})$ and such that $T_{u'}$ only contains leaves in state 0, $u$ is called a \emph{maximal $\mathit{0}$-node}. Finally, if $u$ is both a descendant of the $\cB$ node and a maximal $\mathit{0}$-node, we refer to $u$ as a \emph{maximal $\0$-node}. 
\end{definition}

As an example, consider tree $T$ and character $f$ depicted in Figure \ref{Fig_1tree}. Then, leaf 3 is a maximal $\0$-node, and the parent of leaves 5 and 6 is a maximal $\mathit{0}$-node but not a maximal $\0$-node (because it is not a descendant of the $\cB$-node).

\begin{remark} \label{remark_maximal0node}
Even though a maximal $\0$-node is defined to be a descendant of the $\cB$-node it cannot be a child of it; this would contradict the definition of the $\cB$-node as the MRCA of all leaves in state 1. Analogously, a maximal $\0$-node cannot be the $\cB$-node itself.
\end{remark}

\begin{lemma}[Death lemma] \label{death_lemma}
Let $T$ be a rooted binary phylogenetic tree and let $f$ be a binary character. Then, in a Dollo-$k$ labeling for $f$, every maximal $\0$-node $u$ has a death event on its incoming edge. In particular, $u$ and all nodes descending from $u$ are assigned state 0. 
\end{lemma}

\begin{proof}
Let $f$ be a binary character that assigns state 1 to at least one leaf (if $f$ assigns state 0 to all leaves, there is no $\cB$-node and thus no maximal $\0$-node to consider). 
Let $v$ be a maximal $\0$-node and let $p$ be the parent of $v$. 
As $v$ is a maximal $\0$-node, $p$ must have a descendant leaf in state 1; otherwise $v$ would not be maximal. Moreover, as $v$ is a descendant of the $\cB$-node, one of the ancestors of $p$ must be the $\cB$-node (note that $p$ cannot be the $\cB$ node as it cannot be the MRCA of all leaves in state 1 as $v$ is a maximal $\0$-node). In particular, $p$ lies on the path between the $\cB$-node and some leaf in state 1. Thus, by Lemma \ref{one_lemma}, $p$ is assigned state 1 in any Dollo-$k$ labeling for $f$. In summary, $p$ is assigned state 1, whereas $v$ is a maximal $0$-node. Thus, at least one death event is required to realize $f$ on $T$. Putting this death event on the edge $(p,v)$ (i.e. assigning state 0 to $v$) and assigning state 0 to all other nodes in the $\0$-clade of $v$ clearly minimizes the number of death events required to realize $f$ on $T$. Thus, in a Dollo-$k$ labeling death events must occur on edges directed into maximal $\0$-nodes, and all nodes descending from maximal $\0$-nodes must be assigned state 0. This completes the proof.
\end{proof}

Based on Lemmas \ref{zero_lemma} -- \ref{death_lemma} we can now prove Theorem \ref{thm_dollo_labeling}.

\begin{proof}[Proof of Theorem \ref{thm_dollo_labeling}]
Let $T$ be a rooted binary phylogenetic tree and let $f$ be a binary character. We first show that there is a unique Dollo-$k$ labeling for $f$. 
In order to do so, we partition the nodes of $T$ into three disjoint sets:
\begin{enumerate}
\item [(i)] nodes not descending from the $\cB$-node.
\item [(ii)] nodes on paths from the $\cB$-node (including the $\cB$-node) to leaves in state 1 (including these leaves).
\item [(iii)] nodes descending from maximal $\0$-nodes (including the maximal $\0$-nodes themselves).
\end{enumerate}
The sets described in (i)--(iii) are clearly disjoint. Moreover, the union of the sets in (i)--(iii) comprises $V(T)$. Set (i) covers all nodes in $V(T) \setminus V(T_{\cB})$ (where $T_{\cB}$ is the subtree of $T$ rooted at the $\cB$-node). Moreover, all nodes in $V(T_{\cB})$ are either part of a path from the $\cB$-node to a leaf in state 1, or they are descending from a maximal $\0$-node. Thus, they are covered by sets (ii) and (iii).

Now, by Lemma \ref{zero_lemma} all nodes in (i) must be assigned state 0 in a Dollo-$k$ labeling for $f$, by Lemma \ref{one_lemma} all nodes in (ii) must be assigned state 1, and by Lemma \ref{death_lemma} all nodes in (iii) must be assigned state 0. In particular, there is no node for which there is any choice in whether it is assigned state 0 or state 1. In other words, the Dollo-$k$ labeling for $f$ is unique.

It remains to show that this Dollo-$k$ labeling is such that all nodes of $T$ that are part of the 1-tree $\cT(f)$ are assigned state 1, whereas all other nodes are assigned state 0. Note that by Lemmas \ref{zero_lemma} -- \ref{death_lemma}, the only nodes of $T$ that are assigned state 1 are those nodes that lie on a path between the $\cB$-node and a leaf in state 1. However, the union of all such paths is precisely the 1-tree, and thus we have $g(v)=1$ for all $v \in V(\cT(f))$ and $g(v)=0$ for $v \in V(T) \setminus V(\cT(f))$ as claimed. This completes the proof.
\end{proof}

We can now directly translate Theorem \ref{thm_dollo_labeling} into a linear time algorithm for finding the unique Dollo-$k$ labeling for a binary character $f$ on a rooted binary phylogenetic tree $T$.

\begin{algorithm}[H]\label{Alg_Dollo_1tree}
\KwIn{Rooted binary phylogenetic tree $T$, binary character $f$}
\KwOut{Dollo-$k$ labeling for $f$}
\DontPrintSemicolon
\caption{Compute Dollo-$k$ labeling} 
Compute the 1-tree $\cT(f)$\;
Assign state 1 to all nodes in $ V(\cT(f))$\;
Assign state 0 to all nodes in $V(T) \setminus V(\cT(f))$\;
\end{algorithm}

\begin{remark} \label{remark_complexity}
All steps in Algorithm \ref{Alg_Dollo_1tree} are linear in the number of nodes of $T$. Note, however, that there are different possibilities for computing the 1-tree. 
\begin{enumerate}
\item [(i)] One possibility is to first compute the $\cB$-node, i.e.  the MRCA of all leaves in state 1. This is, for example, possible in the software package BEAST 2 \citep{Bouckaert2019} via the method \texttt{MRCAPrior.getCommonAncestor()} and is also implemented in the BioPerl \citep{Stajich2002} package \\ \texttt{Bio::Tree::TreeFunctionsI} via the method \texttt{get\_lca()}. Once the $\cB$-node has been computed, the 1-tree can be constructed by taking the union of paths between the $\cB$-node and leaves in state 1. 
\item [(ii)] Alternatively, the 1-tree can be computed by considering tree $T$ as undirected, fixing one leaf in state 1, say $x$, and then computing all paths between $x$ and any other leaf in state 1. The union of these paths will then comprise the 1-tree, and this computation is again linear in the number of nodes of $T$.
\end{enumerate}
An implementation of Algorithm \ref{Alg_Dollo_1tree} can be found in the \texttt{DolloAnnotator} app in the Babel package for BEAST 2 \citep{Bouckaert2019}.
\end{remark}

Theorem \ref{thm_dollo_labeling} and the fact that the 1-tree can be computed in linear time immediately leads to the following corollary.

\begin{corollary}\label{alg_linear}
Let $T$ be a rooted binary phylogenetic tree and let $f$ be a binary character. Then, Algorithm \ref{Alg_Dollo_1tree} produces the Dollo-$k$ labeling for $f$ in linear time.
\end{corollary}

Note that it is not only possible to find the Dollo-$k$ labeling for a binary character $f$ by considering the 1-tree, the 1-tree also allows us to easily compute the Dollo score $k(f,T)$. 

\begin{theorem} \label{prop_dollo_score_1tree}
Let $T$ be a rooted binary phylogenetic tree, let $f$ be a binary character, and let $\cT(f)$ be the 1-tree.
\begin{enumerate}[\rm (i)]
\item If $f$ assigns state 1 to at most one leaf, $k(f,T)=0$.
\item If $f$ assigns state 1 to at least two leaves, $k(f,T)$ corresponds to the number of degree-2 nodes in the 1-tree minus 1, i.e.
$$ k(f,T) = | \{ v \in V(\cT(f)): deg(v)=2\}|-1.$$
\end{enumerate}
\end{theorem}

\begin{proof}
Let $T$ be a rooted binary phylogenetic tree and let $f$ be a binary character.
\begin{enumerate}[\rm (i)]
\item If $f$ assigns state 1 to no leaf (i.e. $f=0\ldots 0$), we clearly have $k(f,T)=0$ (see also proof of Proposition \ref{dollo_0_n-2}). Moreover, if $f$ assigns state 1 to precisely one leaf of $T$, say $x$, $f$ can be realized by one birth event (on the edge incident to leaf $x$) and no death event. In particular, $k(f,T)=0$.
\item Now assume that $f$ assigns state 1 to at least two leaves and let $\cT(f)$ be the 1-tree. As $T$ does not contain degree-2 nodes by definition, the degree-2 nodes in $\cT(f)$ must correspond to edges of $T$ that `connect' $\cT(f)$ to the rest of $T$. In particular, the root of $\cT(f)$ (i.e. the degree-2 node of $\cT(f)$ which is the $\cB$-node) is connected to the rest of $T$ via an edge on which the birth event takes place. All other degree-2 nodes of $\cT(f)$ correspond to edges leading to maximal $\0$-nodes in $T$ and thus (by Lemma \ref{death_lemma}) to edges on which death events take place. By Theorem \ref{thm_dollo_labeling} these are the only death events required to realize $f$ on $T$ (as assigning state 1 to all nodes of $\cT(f)$ and state 0 to all remaining nodes of $T$ yields the unique Dollo-$k$ labeling for $f$ and as there cannot be death events above the $\cB$-node). In particular, $k(f,T)=|\{v \in V(\cT(f)): deg(v)=2\}|-1$ as claimed (note that we subtract 1 as the $\cB$-node is a degree-2 node of $\cT(f)$ that does not correspond to a death event but to the birth event). This completes the proof.
\end{enumerate}
\end{proof}
 
The proof of Part (ii) of Theorem \ref{prop_dollo_score_1tree} directly implies the following corollary.

\begin{corollary}\label{Dollo_count_lemma}
Let $T$ be a rooted binary phylogenetic tree and let $f$ be a binary character. Then, the Dollo score $k(f,T)$ equals the number of maximal $\0$-nodes in $T$.
\end{corollary}

Theorem \ref{prop_dollo_score_1tree} can also be used to decide in linear time whether a binary character $f$ is persistent on a given rooted binary phylogenetic tree $T$ (i.e. to decide whether $f$ is a Dollo-0 or a Dollo-1 character). 

\begin{corollary}\label{cor_1tree_persistent}
Let $T$ be a rooted binary phylogenetic tree, let $f$ be a binary character, and let $\cT(f)$ be the 1-tree. 
Then, $f$ is persistent on $T$ if and only if $\cT(f)$ contains at most two degree-2 nodes.
\end{corollary}

\begin{proof}
The statement is a direct consequence of Theorem \ref{prop_dollo_score_1tree}. If $f$ is persistent, $k(f,T) \leq 1$, and thus $\cT(f)$ contains at most 2 degree-2 nodes, namely the root and possibly one vertex that is incident to a loss edge. If, on the other hand,  $\cT(f)$ contains at most two degree-2 nodes, $k(f,T) \leq 1$, and thus $f$ is persistent. 
\end{proof}

\section{Links and contrasts between persistent characters and general Dollo-$k$ characters}

\subsection{Dollo parsimony and Fitch parsimony} \label{Sec_Fitch}
Corollary \ref{cor_1tree_persistent} implies that in order to decide whether a binary character $f$ is persistent or not it suffices to compute the 1-tree for $f$ and count the number of degree-2 nodes in it. Note that a different approach for deciding whether a binary character $f$ is persistent was recently introduced in \citep{Wicke2018} (cf. Theorem 1 therein). This approach uses a connection between maximum parsimony, the Fitch algorithm, and persistent characters. More precisely, it was shown that a character $f$ with $l(f,T) \leq 1$ is guaranteed to be persistent, whereas a character $f$ with $l(f,T) > 2$ cannot be persistent. Moreover, for a character $f$ with $l(f,T)=2$ the question whether $f$ is persistent or not solely depends on how the ancestral state sets found by the Fitch algorithm are distributed across the tree.
Additionally, it was shown in \citep{Wicke2018} that the unique `persistent' extension for a persistent character $f$ (i.e. its unique Dollo-0 or Dollo-1 labeling, respectively) is always also guaranteed to be a most parsimonious extension (Proposition 1 and Lemma 4 in \citep{Wicke2018}). While we have already seen that this is not the case for general Dollo-$k$ characters (Figure \ref{Fig_Tree}), in the following we show that the difference between the parsimony score and the Dollo score can even be made arbitrarily large. Thus, Dollo parsimony and Fitch parsimony are in general very different. However, the parsimony score can be bounded from above by considering the Dollo score of a character $f$ and its inverted character $\bar{f}$.

\begin{proposition}\label{prop_dollo_pars_bound}
Let $T$ be a rooted binary phylogenetic tree, let $f$ be a binary character, and let $\bar{f}$ be its inverted version. Then,
$$l(f,T) \leq \min \{k(f,T), k(\bar{f},T) \} +1.$$
\end{proposition}

\begin{proof}
First, if $f=0 \ldots 0$, we have $l(f,T)=0 = k(f,T) = k(\bar{f},T)$ and the statement trivially holds. Thus, assume now that $f \neq 0 \ldots 0$.
By definition $k(f,T)$, respectively $k(\bar{f},T)$, is the minimum number of $1 \rightarrow 0$ transitions required to realize $f$, respectively $\bar{f}$, on $T$. Additionally, these $1 \rightarrow 0$ transitions are preceded by precisely one $0 \rightarrow 1$ transition on an edge $e \in E(T) \setminus \{(\rho',\rho)\}$ or on the root edge $(\rho',\rho)$. 
In any case, the Dollo labeling for $f$ is an extension of $f$ that induces a changing number of $k(f,T)+1$, and analogously the Dollo labeling for $\bar{f}$ is an extension of $\bar{f}$ that induces a changing number of $k(\bar{f},T)+1$. Now, as $l(f,T)$ is defined as the minimum changing number over \emph{all} extensions of $f$ (but without taking into account a potential change on the root edge), we clearly have 
$$l(f,T) \leq \begin{cases}
k(f,T) & \parbox[t]{.7\textwidth}{ if the Dollo labeling for $f$ induces a $0 \rightarrow 1$  transition on edge  $(\rho', \rho)$,} \\
 k(f,T)+1 &\text{ else}.
\end{cases}$$
Analogously, 
$$l(\bar{f},T) \leq \begin{cases}
k(\bar{f},T) &\parbox[t]{.7\textwidth}{ if the Dollo labeling for $\bar{f}$ induces a $0 \rightarrow 1$ transition on edge $(\rho', \rho)$,} \\
 k(\bar{f},T)+1 &\text{ else}.
\end{cases}$$
However, as $l(f,T) = l(\bar{f},T)$, this directly implies that $l(f,T) \leq \min\{k(f,T), d(\bar{f},T)\}+1$. This completes the proof.
\end{proof}

While $l(f,T)-1$ is a lower bound for the Dollo scores of $f$ and $\bar{f}$, the absolute difference between the parsimony score and the Dollo score can be made arbitrarily large as we will show in the following proposition. Thus, unlike in the case of persistent characters, there is no close relationship between Fitch parsimony and Dollo parsimony.

\begin{proposition}\label{Prop_Difference}
Let $T$ be a rooted binary phylogenetic tree, let $f$ be a binary character, and let $\bar{f}$ be its inverted version. Then both the absolute differences between $l(f,T)$ and $k(f,T)$ as well as between $l(f,T)$ and $k(\bar{f},T)$ can be made arbitrarily large.
\end{proposition}

\begin{proof}
Consider the rooted binary phylogenetic tree $T$ with $n=6+a+b$ (for some $a,b \in \mathbb{N}_{\geq 1}$) leaves together with the binary character $f$ and its inverted character $\bar{f}$ depicted in Figure \ref{Fig_f_fbar}. Here, we have $l(f,T)=l(\bar{f},T)=3$ (regardless of $n$), whereas $k(f,T)=b+3$ and $k(\bar{f},T)=a+2$. Thus, both $|k(f,T)-l(f,T)|=b$ and $|k(\bar{f},T)-l(f,T)|=|a-1|$ can be made arbitrarily large by increasing $a$ and $b$, respectively. This completes the proof.
\end{proof}

\begin{figure}[htbp]
    \centering
    \includegraphics[scale=0.25]{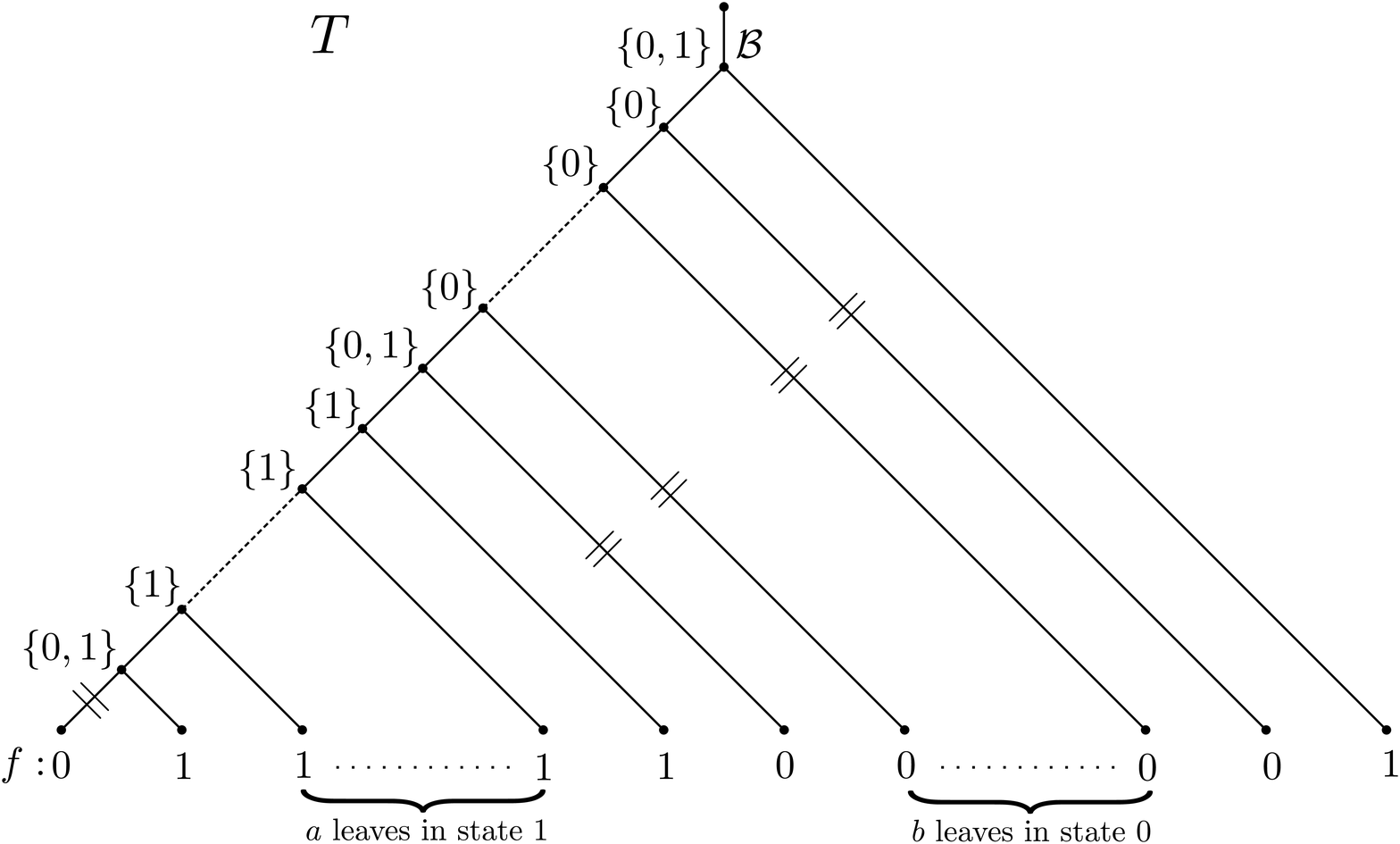} \\
    \includegraphics[scale=0.25]{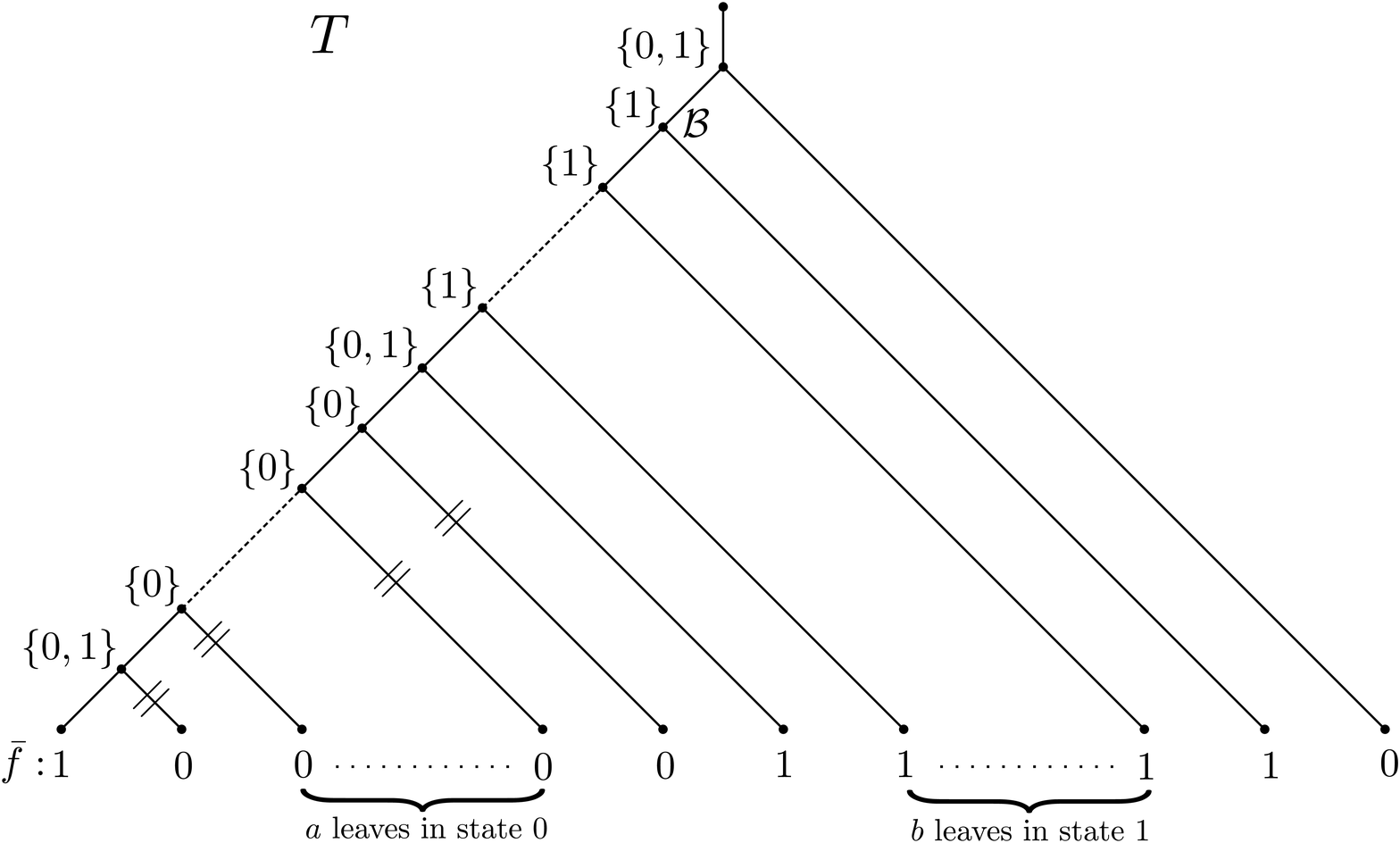}
    \caption{Rooted binary phylogenetic tree with $n=6+a+b$ leaves (leaf labels are omitted) and binary characters $f$ and $\bar{f}$. When applying the first phase of the Fitch algorithm to $f$ and $\bar{f}$, the union of $\{0\}$ and $\{1\}$ is taken three times (indicated by $\{0,1\}$ sets assigned to three nodes of $T$), respectively, and thus $l(f,T)=l(\bar{f},T)=3$. On the other hand, $k(f,T)=b+3$, and $k(\bar{f},T)=a+2$ as there are $b+3$, respectively $a+2$, death events required to realize $f$, respectively $\bar{f}$, on $T$ (namely, on all edges incident to leaves in state 0).}
    \label{Fig_f_fbar}
\end{figure}

\subsection{Relationship between the Sackin index and the number of Dollo-$k$ characters} \label{section_sackin}
A second striking difference between persistent characters and general Dollo-$k$ characters can be observed when counting the number of persistent characters, respectively general Dollo-$k$ characters, for a given tree $T$. 

For persistent characters, i.e. the union of Dollo-0 and Dollo-1 characters, it has been observed in \citep{Wicke2018} that there is a direct relationship between the so-called Sackin index \citep{Sackin1972}, an index of tree balance, and the number of persistent characters (Theorem 2 in \citep{Wicke2018}). In particular, the number of persistent characters of a rooted binary phylogenetic tree $T$ can directly be obtained by calculating the Sackin index of $T$, a task easy to accomplish.

Therefore, recall that the Sackin index of a rooted binary phylogenetic tree is defined as follows.
\begin{definition}[\citep{Sackin1972}] \label{Def_Sackin}
Let $T$ be a rooted binary phylogenetic tree. Then, its \emph{Sackin index} is defined as 
$$ \mathcal{S}(T) = \sum\limits_{u \in \mathring{V}(T)} n_u,$$
where $\mathring{V}(T)$ denotes the set of internal nodes of $T$ and $n_u$ denotes the number of leaves of the subtree of $T$ rooted at $u$.
\end{definition}

Note that the higher the Sackin index of a tree $T$, the more imbalanced it is. 
Now, in \citep{Wicke2018} the following relationship between the Sackin index and the number of persistent characters was  established.

\begin{theorem}[adapted from Theorem 2 in \citep{Wicke2018}] \label{Thm_Sackin}
Let $T$ be a rooted binary phylogenetic tree with $n$ leaves. Let $\cP(T)$ denote the number of persistent characters of $T$. Then, 
$$ \cP(T) = 2 \mathcal{S}(T) - 2n + 4.$$
\end{theorem}

Note that as persistent characters are the union of Dollo-0 and Dollo-1 characters. It is easy to see that \emph{every} rooted binary phylogenetic tree $T$ with $n$ leaves induces $2n$ Dollo-0 characters (we will formally show this in Theorem \ref{Thm_sum_ik}). Thus, the correspondence between the Sackin index and the number of persistent characters is really a correspondence between the Sackin index and the number of Dollo-1 characters, and we have that the number of Dollo-1 characters equals $\cP(T)-2n = 2  \mathcal{S}(T)+4$. In particular, the more imbalanced a tree $T$ is (i.e. the higher its Sackin index), the more Dollo-1 characters it induces. However, this `trend' does not continue for $k \geq 2$. There neither seems to be a direct relationship between the Sackin index and the number of Dollo-2 characters nor between the Sackin index and the cardinality of the union of Dollo-0, Dollo-1 and Dollo-2 characters. As an example, consider trees $T_1, T_2$, and $T_3$ on 5 leaves depicted in Figure \ref{Fig_Balance}. We have, $\mathcal{S}(T_1)=12, \mathcal{S}(T_2)=13$ and $\mathcal{S}(T_3)=14$. In particular, $T_3$ has more Dollo-1 characters than $T_2$, and $T_2$ has more Dollo-1 characters than $T_1$. However, $T_3$ has fewer Dollo-2 characters than $T_1$ and $T_2$. Moreover, if we consider the union of Dollo-0, Dollo-1 and Dollo-2 characters, $T_3$ has fewer such characters than $T_2$ even though it is more imbalanced. Moreover, this example suggests that the highest $k$ such that $T$ induces Dollo-$k$ characters is not directly related to the Sackin index, either. For example, while $T_2$ is less balanced than $T_1$ and more balanced than $T_3$, it does not induce Dollo-3 characters, whereas \emph{both} $T_1$ and $T_3$ do. In this case, this is due to the fact that both $T_1$ and $T_3$ are semi-caterpillar trees. We will, however, see in Proposition \ref{Existence_Dollon-2} that there is a connection between the number of Dollo-$(n-2)$ characters and the shape of a tree. 
 
\begin{figure}[htbp]
    \centering
    \includegraphics[scale=0.25]{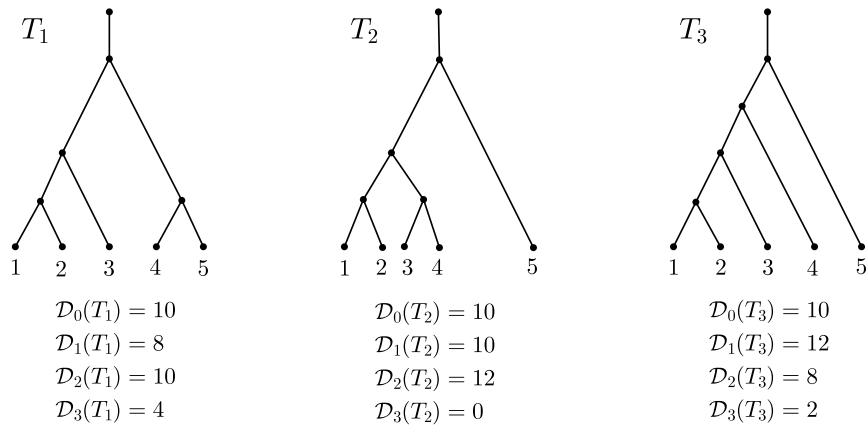}
    \caption{Rooted binary phylogenetic trees $T_1, T_2$ and $T_3$ with leaf set $\{1, \ldots, 5\}$. The trees are ordered according to their balance (from balanced to imbalanced according to the Sackin index). More precisely, we have $\mathcal{S}(T_1)=12$, $\mathcal{S}(T_2)=13$ and $\mathcal{S}(T_3)=14$. Note that for $n=5$, there are $2^5=32$ binary characters, and the numbers $\cD_k(T_i)$ indicate how many of them are Dollo-$k$ characters for $T_i$ with $i \in \{1,2,3\}$ and $k=0, \ldots, 3$.}
    \label{Fig_Balance}
\end{figure}

\section{Counting Dollo-$k$ characters}
The question of how many characters are persistent on given tree $T$ can easily be answered by calculating the Sackin index of $T$ and applying Theorem \ref{Thm_Sackin}. 
For general $k$, however, it turns out that determining the number of Dollo-$k$ characters is much more involved, and we need to introduce further definitions and notations to do so.

\subsection{Recursively computing the number of Dollo-$k$ characters}
In the following let $\cD_k(T)$ denote the number of Dollo-$k$ characters of $T$. 
\red{Note that if $T$ contains $n$ leaves, then $\sum_{k \in \mathbb{N}_{\geq 0}} \cD_k(T)=2^n$.
To see this, notice that there are $2^n$ binary characters of length $n$. Moreover, for each character $f$, the Dollo score $k(f,T)$ of $f$ on $T$ is a unique non-negative integer in $\{0, \ldots, n-2\}$ (cf. Proposition \ref{dollo_0_n-2}). In particular, each of the $2^n$ binary characters contributes to precisely one of the $\cD_k(T)$-values for $k \in \mathbb{N}_{\geq 0}$ (if $k(f,T) = k_f$, then $f$ is a Dollo-$k_f$ character and contributes to $\mathcal{D}_{k_f}(T)$). Thus, if we sum over the number of Dollo-$k$ characters of $T$ and range over all possible values of $k$, we get the total number of binary characters.}
Determining the $\cD_k(T)$-values themselves, however, is much more involved.

Before we can state the main theorem of this section showing how to calculate $\cD_k(T)$ recursively, we require two technical definitions.

\begin{definition}[Extended independent node set] \label{def_extendedindependent}
Let $T$ be a rooted binary phylogenetic tree with root $\rho$. A set of $k$ nodes $\{u_1, \ldots, u_k\}$ of $T$ such that
    \begin{enumerate}[\rm (i)]
        \item $u_s \neq \rho \, \forall s$,
        \item $u_s \text{ is not an ancestor of } u_t \, \forall s \neq t,$
        \item $u_s \text{ and } u_t \text{ are not siblings } \forall s \neq t$
    \end{enumerate}
is called an \emph{extended independent node set of size $k$ for $T$}. We use $E_k(T)$ to denote the set of all extended independent node sets of size $k$ for $T$, i.e.
\begin{align}
E_k(T) = \{ \{u_1, \ldots, u_k\} \, | \, &u_s \neq \rho \, \forall s, \nonumber \\
&u_s \text{ is not an ancestor of } u_t \, \forall s \neq t,\\
&u_s \text{ and } u_t \text{ are not siblings } \forall s \neq t \}, \nonumber 
\end{align}
and let $e_k(T) \coloneqq \vert E_k(T) \vert$ denote its cardinality. For technical reasons, we set $e_0(T) \coloneqq 1$.
\end{definition} 
Note that it is totally possible that $e_k(T)=0$. In particular, if $T$ contains only one leaf, then $e_k(T)=0$ for all $k \geq 1$. \\

Next to extended independent node sets, we also require the notion of \emph{independent node sets}, which are defined as follows.
\begin{definition}[Independent node set] \label{Def_independentnodeset}
Let $T$ be a rooted binary phylogenetic tree with root $\rho$. A set of $k$ nodes $\{u_1, \ldots, u_k\}$ of $T$ is called an \emph{independent node set of size $k$ for $T$} if it is an extended independent node set of size $k$ for $T$ that next to conditions (i)--(iii) (cf. Definition \ref{def_extendedindependent}) additionally satisfies
\begin{enumerate}
    \item [(iv)] $u_s$ is not a child of $\rho \, \forall \, s.$
\end{enumerate}
We use $I_k(T)$ to denote the set of independent node sets of size $k$ of $T$, i.e.
\begin{align}
I_k(T) = \{ \{u_1, \ldots, u_k\} \, | \, &u_s \neq \rho \text{ and $u_s$ is not a child of $\rho \, \forall s$,} \nonumber \\
&u_s \text{ is not an ancestor of } u_t \, \forall s \neq t, \\
&u_s \text{ and } u_t \text{ are not siblings } \forall s \neq t \}. \nonumber
\end{align}
Moreover, we use $i_k(T) \coloneqq \vert I_k(T) \vert$ to denote the number of independent node sets of size $k$ for $T$. For technical reasons, $i_0(T) \coloneqq 1$.
\end{definition}

\noindent Based on this, we can now state the main theorem of this section. This theorem states that the number of Dollo-$k$ characters a rooted binary phylogenetic tree $T$ induces can be calculated by considering extended independent node sets.
\begin{theorem}\label{thm_counting}
Let $T$ be a rooted binary phylogenetic tree with $n \geq 2$ leaves and let $\mathcal{D}_k(T)$ denote the number of Dollo-$k$ characters for $T$. Then,
\begin{enumerate}[\rm (i)]
    \item For $k=0$, we have $\cD_0(T)=2n$.
    \item For $k \geq 1$, we have:
    \begin{align}
    \cD_k(T) &= \sum\limits_{\substack{u \in V(T) \setminus \{\rho'\} \\ n_u \geq 2}}  \, \sum\limits_{i=0}^k e_i(T_u^1) \cdot e_{k-i}(T_u^2),
    \end{align}
    where $T_u=(T_u^1, T_u^2)$ denotes the standard decomposition of $T_u$ and where
    $$e_i(T_u^1) = \begin{cases}
    1, &\text{if $i=0$} \\
    0, &\parbox[t]{.5\textwidth}{if $i > 0$ and $T_u^1$ contains precisely one leaf.}  \\
    \sum\limits_{j=0}^i \left(e_j(T_u^{11}) \cdot e_{i-j}(T_u^{12}) \right) \\
    \quad + e_{i-1}(T_u^{11}) + e_{i-1}(T_u^{12}), &\parbox[t]{.5\textwidth}{if $i > 0$ and $T_u^1=(T_u^{11},T_u^{12})$ contains at least two leaves,}
    \end{cases} $$
    and where
    $$e_{k-i}(T_u^2) = \begin{cases}
    1 &\text{if $k-i=0$} \\
    0 &\parbox[t]{.4\textwidth}{if $k-i > 0$ and $T_u^2$ contains precisely one leaf.}  \\
    \sum\limits_{j=0}^{k-i} \left(e_j(T_u^{21}) \cdot e_{k-i-j}(T_u^{22}) \right) \\
    \quad + e_{k-i-1}(T_u^{21}) + e_{k-i-1}(T_u^{22}), &\parbox[t]{.4\textwidth}{if $k-i > 0$ and $T_u^2=(T_u^{21},T_u^{22})$ contains at least two leaves.}
    \end{cases} $$
\end{enumerate}
\end{theorem}

In order to prove Theorem \ref{thm_counting}, we require several technical lemmas. The first one establishes a connection between the number of Dollo-$k$ characters (with $k \geq 1$) for a rooted binary phylogenetic tree $T$ and the number of independent node sets for $T$ and its subtrees.

\begin{lemma} \label{Thm_sum_ik}
Let $T$ be a rooted binary phylogenetic tree $T$ and let $\mathcal{D}_k(T)$ be the number of Dollo-$k$ characters for $T$ and $k \geq 1$. Let $i_k(T_u)$ denote the number of independent node sets of size $k$ for a subtree $T_u$ of $T$. Then,
\begin{align*}
\mathcal{D}_k(T) &= \sum\limits_{u \in V(T) \setminus \{\rho'\}} i_k(T_u).
\end{align*}
\end{lemma}

\begin{proof}
By Corollary \ref{Dollo_count_lemma}, the Dollo score $k(f,T)$ equals the number of maximal $\0$-nodes in $T$ (which are by definitions descendants of the $\cB$-node).

\noindent Recall that maximal $\0$-nodes have the following properties:
\begin{enumerate}[\rm (a)]
   \item All maximal $\0$-nodes have the $\cB$-node as an ancestor (by definition) but they cannot be children of the $\cB$-node or the $\cB$-node itself (cf. Remark \ref{remark_maximal0node}).
   \item If $u$ and $v$ are both maximal $\0$-nodes, $u$ cannot be an ancestor of $v$ or vice versa (otherwise, one of them would not be a maximal $\0$-node).
   \item If $u$ and $v$ are both maximal $\0$-nodes, they cannot be siblings (if $u$ and $v$ were siblings, they would not be maximal $\0$-nodes but their parent would be a maximal $\0$-node).
\end{enumerate}

\noindent Comparing properties (a)--(c) with Definition \ref{Def_independentnodeset}, this implies that every set of $k$ maximal $\0$-nodes forms an independent node set of size $k$ for the subtree of $T$ rooted at the $\cB$-node. Conversely, any independent node set of size $k$ for $T_{\cB}$ (i.e. for the subtree of $T$ rooted at $\cB$ can be considered as a set of $k$ maximal $\0$-nodes). \\

\noindent Now let $k \geq 1$. In order to count the number of Dollo-$k$ characters $T$ induces, we need to count the number of ways to choose one node of $T$ as the $\cB$-node and $k$ additional nodes as maximal $\0$-nodes. 
Each such choice of $k+1$ nodes in total yields precisely one Dollo-$k$ character (namely the one where all leaves not descending from the $\cB$-node are assigned state 0, all leaves descending from the $\cB$-node but not from any of the $k$ maximal $\0$-nodes are assigned state 1, and all remaining leaves are assigned state 0). Now, if we fix the $\cB$-node, there are $i_k(T_\cB)$ many independent node sets of size $k$ for $T_\cB$ and thus $i_k(T_\cB)$ sets of $k$ maximal $\0$-nodes below the $\cB$-node. 
As all nodes of $T$ except for $\rho'$ can potentially be the $\cB$-node, summing over all nodes $u \in V(T) \setminus \{\rho'\}$ and adding up the number of independent node sets of size $k$, respectively, yields the number $\cD_k(T)$ of Dollo-$k$ characters for $T$. 
Thus, $\cD_k(T) =  \sum_{u \in V(T) \setminus \{\rho'\}} i_k(T_u)$ as claimed. This completes the proof.
\end{proof}

\begin{example}
Consider the rooted binary phylogenetic tree $T$ with 5 leaves depicted in Figure \ref{Fig_IndependentSets} and let $k=3$. Then, by Lemma \ref{Thm_sum_ik}
\begin{align*}
\cD_3(T) &= \sum\limits_{u \in V(T) \setminus \{\rho'\}} i_3(T_u) = i_3(T_{\rho}) + i_3(T_a) + i_3(T_b) + i_3(T_c) + i_3(T_1) + \ldots + i_3(T_5) \\
&= \left\vert \left\lbrace \{1,3,4\}, \, \{1,3,5\}, \, \{2,3,4\}, \, \{2,3,5\}  \right\rbrace \right\vert + \left\vert \emptyset \right\vert + \ldots + \left\vert \emptyset \right\vert \\
&\quad \text{ (as there are no independent node sets of size 3 for } T_a, T_b, T_c, T_1, \ldots, T_5) \\
&= 4.
\end{align*}
Thus, $T$ induces four Dollo-3 characters. These are $f_1=01001$, $f_2=01010$, $f_3=10001$, and $f_4=10010$.
Analogously, $\mathcal{D}_1(T) = 8$ and $\mathcal{D}_2(T)=10$. Moreover, $\mathcal{D}_0(T) = 2n = 10$.
Note that $\mathcal{D}_k(T) = 0$ for $k > 3=n-2$ by Proposition \ref{dollo_0_n-2}.
\end{example}

\begin{figure}[h!]
	\centering
	\includegraphics[scale=0.25]{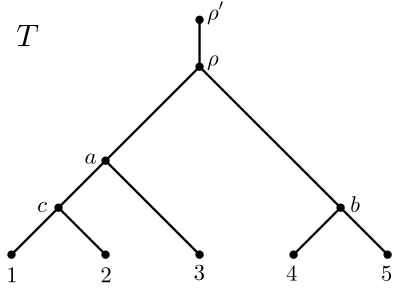}
	\caption{Rooted binary phylogenetic tree $T$ with leaf set $\{1,\ldots,5\}$. Exemplarily, $I_3(T_{\rho}) = \left\lbrace \{1,3,4\}, \, \{1,3,5\}, \, \{2,3,4\}, \, \{2,3,5\} \right\rbrace$ and $i_3(T_{\rho})=4$.}
	\label{Fig_IndependentSets}
\end{figure}

The next lemma states that the number of independent node sets for a rooted binary phylogenetic tree $T$ can be obtained by calculating the numbers of extended independent node sets for its maximal pending subtrees.

\begin{lemma} \label{cor_number_independent_extended}
Let $T = (T_a, T_b)$ be a rooted binary phylogenetic tree with $n \geq 2$ leaves and root $\rho$, where $T_a$ and $T_b$ are the two maximal pending subtrees of $T$ rooted at the children $a$ and $b$ of $\rho$. Then, 
\begin{align*}
i_k(T) = \sum\limits_{i=0}^k  e_i(T_a) \cdot e_{k-i}(T_b),
\end{align*}
where $e_0(T_a) = e_0(T_b)=1$. 
\end{lemma}

\begin{proof}
In order to prove Lemma \ref{cor_number_independent_extended}, we first show that a set $\{u_1, \ldots, u_k\}$ of nodes is an independent node set of size $k$ for $T$ if and only if it is the union of an extended independent node set of size $i$ for $T_a$ and an extended independent node set of size $k-i$ for $T_b$ (where $i \in \mathbb{N}_0$). \\

\noindent Let $\mathcal{I}_k = \{u_1, \ldots, u_k\} \in I_k(T)$ be an independent node set of size $k$ for $T$.
Now, suppose that $i$ of the $k$ nodes in $\mathcal{I}_k$, without loss of generality $u_1, \ldots, u_i$, are in $T_a$ and $k-i$ nodes, without loss of generality $u_{i+1}, \ldots, u_k$, are in $T_b$, i.e.
$$ \mathcal{I}_k = \{u_1, \ldots, u_k\} = \underbrace{\{u_1, \ldots, u_i\}}_{\in V(T_a)} \, \cup \, \underbrace{\{u_{i+1}, \ldots, u_k\}}_{\in V(T_b)}.$$
Then, as $\mathcal{I}_k$ by definition does not contain nodes $a$ and $b$, in particular the set $\{u_1, \ldots, u_i\}$ does not contain $a$ and is thus an extended independent node set of size $i$ of $T_a$, i.e. $\{u_1, \ldots, u_i\} \in E_i(T_a)$. Analogously, the set $\{u_{i+1}, \ldots, u_k\}$ is an extended independent node set of size $k-i$ of $T_b$, i.e. $\{u_{i+1}, \ldots, \} \in E_j(T_b)$. In particular, $\mathcal{I}_k$ is the union of an extended independent node set of size $i$ of $T_a$ and an extended independent node set of size $k-i$ of $T_b$. \\

\noindent Now, suppose that $\mathcal{E}^a_i = \{a_1, \ldots, a_i\} \in E_i(T_a)$ is an arbitrary extended independent node set of size $i$ for $T_a$ and $\mathcal{E}^b_{k-i} = \{b_1, \ldots, b_{k-i}\} \in E_{k-i}(T_b)$ is an arbitrary extended independent node set of size $k-i$ for $T_b$. 
Consider $\mathcal{I} = \mathcal{E}_i^a \, \cup \, \mathcal{E}_{k-i}^b = \{a_1, \ldots, a_i, b_1, \ldots, b_{k-i}\}$. Then, $\mathcal{I}$ is an independent node set of size $k$ for $T$. To see this, consider the following:
	\begin{itemize}
	\item As $i+(k-i)=k$ and as $\mathcal{E}_i^a \, \cap \, \mathcal{E}_{k-i}^b = \emptyset$, $\mathcal{I}$ contains $k$ nodes (i.e. it has the correct size).
	\item As $\mathcal{E}_i^a$ by definition does not contain node $a$ and $\mathcal{E}_{k-i}^b$ by definition does not contain node $b$ and as neither $\mathcal{E}_i^a$ nor $\mathcal{E}_{k-i}^b$ can contain node $u$ (as $u$ is not contained in $T_a$ or $T_b$), we have $a, b, u \notin \mathcal{I}$.
	\item As $\mathcal{E}_i^a$ and $\mathcal{E}_{k-i}^b$ are extended independent node sets, 
		\begin{itemize}
		\item $a_s$ is not an ancestor of $a_t$ and $b_s$ is not an ancestor of $b_t$ for all $s \neq t$,
		\item neither $a_s$ and $a_t$ nor $b_s$ and $b_t$ are siblings for all $s \neq t$.
		\end{itemize}
		Moreover, a node $a' \in \mathcal{E}_i^a$ cannot be an ancestor of a node $b' \in \mathcal{E}_{k-i}^b$ in $T_u$ and vice versa (since $a'$ is in $T_a$ and $b'$ is in $T_b$). 
		Similarly, $a'$ and $b'$ cannot be siblings in $T_u$ (note that nodes $a$ and $b$ are siblings in $T_u$, but $a \notin \mathcal{E}_i^a$ and $b \notin \mathcal{E}_{k-i}^b$).
	\end{itemize}
Thus, $\mathcal{I}$ is an independent node set of size $k$ for $T_u$. \\

\noindent In summary, every independent node set of size $k$ of a rooted binary phylogenetic tree $T$ with $n \geq 2$ leaves corresponds to a union of extended independent node sets of suitable size of its two maximal pending subtrees and each such union of extended independent node sets yields an independent node set of size $k$ for $T$. This directly implies $$i_k(T) = \sum\limits_{i=0}^k  e_i(T_a) \cdot e_{k-i}(T_b)$$ as claimed.
\end{proof}

\begin{example} \label{Example_Independent_Extended}
Consider the rooted binary phylogenetic tree $T$ on 5 leaves depicted in Figure \ref{Fig_IndependentSets}. Here, we have $I_3(T) = \left\lbrace \{1,3,4\}, \, \{1,3,5\}, \, \{2,3,4\}, \, \{2,3,5\} \right\rbrace$, i.e. $i_3(T)=4$. 
Using Lemma \ref{cor_number_independent_extended} this can be verified as follows:
For the two maximal pending subtrees $T_a$ and $T_b$ of $T$, we have: $e_0(T_a)=e_0(T_b)=1$ (by definition), $e_1(T_a)=4$ (as $E_1(T_a) = \left\lbrace \{1\}, \{2\}, \{3\}, \{c\} \right\rbrace$), $e_1(T_b)=2$ (as $E_1(T_b) = \left\lbrace \{4\}, \{5\} \right\rbrace$), $e_2(T_a)=2$ (as $E_2(T_a) = \left\lbrace \{1,3\}, \{2,3\} \right\rbrace$) and $e_i(T_a) = e_j(T_b) = 0$ for $i > 2$ and $j > 1$. Thus, using Lemma \ref{cor_number_independent_extended}, we have
\begin{align*}
i_3(T) &= \sum\limits_{i=0}^3 e_i(T_a) \cdot e_{3-i}(T_b) \\
&= e_0(T_a) \cdot e_3(T_b) + e_1(T_a) \cdot e_2(T_b) + e_2 (T_a) \cdot e_1(T_b) + e_3(T_a) \cdot e_0(T_b) \\
&= 1 \cdot 0 + 4 \cdot 0 + 2 \cdot 2 + 0 \cdot 1 = 4.
\end{align*}
\end{example}

Lemma \ref{cor_number_independent_extended} implies that the calculation of the number of independent node sets of size $k$ of a rooted tree $T$ can be reduced to the calculation of the number of extended independent node sets of a certain size of its two maximal pending subtrees. The next lemma shows how the latter can be computed.

\begin{lemma} \label{extended_number_recursion}
Let $T = (T_a, T_b)$ be a rooted binary phylogenetic tree with $n \geq 2$ leaves and root $\rho$, where $T_a$ and $T_b$ are the two maximal pending subtrees of $T$ rooted at the children $a$ and $b$ of $\rho$. Then, 
\begin{align*}
e_k(T) = \sum\limits_{i=0}^k  e_i(T_a) \cdot e_{k-i}(T_b) + e_{k-1}(T_a) + e_{k-1}(T_b),
\end{align*}
where $e_0(T_a) = e_0(T_b)=1$. 
\end{lemma}

\begin{proof}
In order to prove Lemma \ref{extended_number_recursion}, we show that 
a set $\{u_1, \ldots, u_k\}$ is an extended independent node set of size $k$ for $T$ if and only if one of the following holds:
	\begin{itemize}
	\item $\{u_1, \ldots, u_k\}$ is the union of an extended independent node set of size $i$ for $T_a$ and an extended independent node set of size $k-i$ for $T_b$ (where $i \in \mathbb{N}_0$);
	\item $\{u_1, \ldots, u_k\}$ is the union of $\{a\}$ and an extended independent node set of size $k-1$ for $T_b$;
	\item $\{u_1, \ldots, u_k\}$ is the union of $\{b\}$ and an extended independent node set of size $k-1$ for $T_a$.
	\end{itemize}
	
\noindent Let $\mathcal{E}_k = \{u_1, \ldots, u_k\} \in E_k(T)$ be an extended independent node set of size $k$ for $T$.
Suppose that $i$ of the $k$ nodes in $\mathcal{E}_k$, without loss of generality $u_1, \ldots, u_k$, are in $T_a$ and $k-i$ nodes, without loss of generality $u_{i+1}, \ldots, u_k$ are in $T_b$, i.e.
$$ \mathcal{E}_k = \{u_1, \ldots, u_k\} = \underbrace{\{u_1, \ldots, u_i\}}_{\in V(T_a)} \, \cup \, \underbrace{\{u_{i+1}, \ldots, u_k\}}_{\in V(T_b)}.$$
First, assume that $a, b \notin \mathcal{E}_k$. Then, by definition, the set $\{u_1, \ldots, u_i\}$ is an extended independent node set of size $i$ for $T_a$, i.e. $\{u_1, \ldots, u_i\} \in E_i(T_a)$. Analogously, the set $\{u_{i+1}, \ldots, u_k\}$ is an extended independent node set of size $k-i$ for $T_b$, i.e. $\{u_{i+1}, \ldots, u_k\} \in E_{k-i}(T_b)$. In particular, $\mathcal{E}_k$ is the union of an extended independent node set of size $i$ for $T_a$ and an extended independent node set of size $k-i$ for $T_a$. 

Now, assume that $a \in \mathcal{E}_k$. This implies that $b \notin \mathcal{E}_k$ (since $a$ and $b$ are siblings in $T$). Moreover, it implies that no other node in $\mathcal{E}_k$ can be contained in $T_a$ (since $a$ is the ancestor of all nodes in $T_a$). Without loss of generality suppose that $u_1 = a$, i.e.
$$ \mathcal{E}_k = \{u_1, \ldots, u_k\} = \underbrace{\{a\}}_{\in V(T_a)} \, \cup \, \underbrace{\{u_2, \ldots, u_k\}}_{\in V(T_b) \setminus\{b\}}.$$
Then, by definition, $\{u_2, \ldots, u_k\}$ is an extended independent node set of size $k-1$ of $T_b$. In particular, $\mathcal{E}_k$ is the union of $\{a\}$ and an extended independent node set of size $k-1$ of $T_b$. 

Analogously, it can be shown that if $b \in \mathcal{E}_k$, $\mathcal{E}_k$ is the union of $\{b\}$ and an extended independent node set of size $k-1$ of $T_a$.\\

\noindent Conversely, first suppose that $\mathcal{E}_i^a = \{a_1, \ldots, a_i\} \in E_i(T_a)$ is an arbitrary extended independent node set of size $i$ for $T_a$ and $\mathcal{E}_{k-i}^b = \{b_1, \ldots, b_{k-i}\} \in E_{k-i}(T_b)$ is an arbitrary extended independent node set of size $k-i$ for $T_b$. 
Consider $\mathcal{E} = \mathcal{E}_i^a \, \cup \, \mathcal{E}_{k-i}^b = \{a_1, \ldots, a_i, b_1, \ldots, b_j\}$. Then, $\mathcal{E}$ is an extended independent node set of size $k$ of $T$, since:
	\begin{itemize}
	\item $\mathcal{E}$ contains $k$ nodes since $i+(k-i)=k$ and $\mathcal{E}_i^a \, \cap \, \mathcal{E}_{k-i}^b = \emptyset$,
	\item $\mathcal{E}$ does not contain node $\rho$ (since neither $\mathcal{E}_i^a$ nor $\mathcal{E}_{k-i}^b$ contain $\rho$),
	\item As $\mathcal{E}_i^a$ and $\mathcal{E}_{k-i}^b$ are extended independent node sets,
		\begin{itemize}
		\item $a_s \in \mathcal{E}_i^a$ is not an ancestor or sibling of $a_t \in \mathcal{E}_i^a$ for all $s \neq t$,
		\item $b_s \in \mathcal{E}_{k-i}^b$ is not an ancestor or sibling of $b_t \in \mathcal{E}_{k-i}^b$ for all $s \neq t$.
		\end{itemize}
		Moreover, a node $a' \in \mathcal{E}_i^a$ cannot be an ancestor of a node $b' \in \mathcal{E}_{k-i}^b$ in $T$ (and vice versa), since $a' \in T_a$ and $b' \in T_b$. Similarly, $a'$ and $b'$ cannot be siblings in $T$ (note that nodes $a$ and $b$ are siblings in $T$, but by definition $a \notin \mathcal{E}_i^a$ and $b \notin \mathcal{E}_{k-i}^b$).
	\end{itemize}
	Thus, $\mathcal{E}$ is an extended independent node set of size $k$ for $T$.
	
	Now, suppose that $\mathcal{E}_{k-1}^b = \{b_1, \ldots, b_{k-1}\} \in E_{k-1}(T_b)$ is an arbitrary extended independent node set of size $k-1$ for $T_b$ and consider $\mathcal{E} = \{a\} \, \cup \, \mathcal{E}_{k-1}^b = \{a, b_1, \ldots, b_{k-1}\}$. Then, again, $\mathcal{E}$ is an extended independent node set of size $k$ for $T_u$ since:
		\begin{itemize}
		\item $\mathcal{E}$ contains $1+(k-1)=k$ nodes,
		\item $\rho \notin \mathcal{E}$ (since $a \neq \rho$ and $\mathcal{E}_{k-1}^b$ cannot contain node $\rho$),
		\item $b_s \in \mathcal{E}_{k-1}^b$ is not an ancestor or sibling of $b_t \in \mathcal{E}_{k-1}^b$. Moreover, $a$ is not an ancestor or sibling of a node $b_i \in \mathcal{E}_{k-1}^b$ (or vice versa). 
		\end{itemize}
	Thus, $\mathcal{E}$ is an extended independent node set of size $k$ for $T$.
	
	Analogously this follows for $\mathcal{E} = \{b\} \, \cup \, \mathcal{E}_{k-1}^a$, where $\mathcal{E}_{k-1}^a \in E_{k-1}(T_a)$ is an arbitrary extended independent node set of size $k-1$ for $T_a$.\\
	
\noindent From this one-to-one correspondence it now directly follows that 
\begin{align*}
e_k(T) = \sum\limits_{i=0}^k  e_i(T_a) \cdot e_{k-i}(T_b)  + e_{k-1}(T_a) + e_{k-1}(T_b)
\end{align*}
as claimed. 
This completes the proof.
\end{proof}

\begin{example}
Consider the rooted binary phylogenetic tree $T$ on 5 leaves depicted in Figure \ref{Fig_IndependentSets}. For its two maximal pending subtrees $T_a$ and $T_b$, we have: $e_0(T_a) = e_0(T_b)=1$ (by definition), $e_1(T_a) = 4$, $e_1(T_b)=2$, $e_2(T_a) = 2$ and $e_i(T_a)=e_j(T_b)=0$ for $i > 2$ and $j > 1$ (see also Example \ref{Example_Independent_Extended}). Thus, using Lemma \ref{extended_number_recursion}, we have
\begin{align*}
e_3(T) &= \sum\limits_{i=0}^3 \left(e_i(T_a) \cdot e_{3-i}(T_b) \right)  + e_2(T_a) + e_2(T_b) \\
&= \left( e_0(T_a) \cdot e_3(T_b) + e_1(T_a) \cdot e_2(T_b) + e_2 (T_a) \cdot e_1(T_b) + e_3(T_a) \cdot e_0(T_b) \right) \\
&\quad + e_2(T_a) + e_2(T_b) \\
&= \left( 1 \cdot 0 + 4 \cdot 0 + 2 \cdot 2 + 0 \cdot 1 \right) + 2 + 0 = 6.
\end{align*}
And indeed, as $E_3(T) = \left\lbrace \{1,3,4\}, \, \{1,3,5\}, \, \{2,3,4\}, \, \{2,3,5\}, \, \{1,3,b\}, \, \{2,3,b\} \right\rbrace$, we have $e_3(T) = 6$.
\end{example}

We are now finally in the position to prove Theorem \ref{thm_counting}.
\begin{proof}[Proof of Theorem \ref{thm_counting}] \leavevmode \\
For Part (i), recall that by Corollary \ref{Dollo_count_lemma}, $k$ equals the number of maximal $\0$-nodes in $T$. Now, for $k=0$, there are no maximal $\0$-nodes. However, every node of $T$ except for $\rho'$ may be the $\cB$-node. As every rooted binary phylogenetic tree with $n$ leaves has $2n-1$ nodes, this yields $2n-1$ Dollo-0 characters containing at least one 1. Additionally, the constant character $f=0 \ldots 0$ (for which there is no $\cB$-node) is a Dollo-0 character. Thus, in total there are $2n$ Dollo-0 characters. \\

\noindent Part (ii) now follows from Lemmas \ref{Thm_sum_ik} -- \ref{extended_number_recursion}. More precisely, by Lemma \ref{Thm_sum_ik}, we have that 
$$  \cD_k(T) = \sum\limits_{u \in V(T) \setminus \{\rho'\}} i_k(T_u) =  \sum\limits_{\substack{u \in V(T) \setminus \{\rho'\} \\ n_u \geq 2}} i_k(T_u),$$
where the last equality follows from the fact that $i_k(T_u)=0$ if $T_u$ contains only one leaf.
Furthermore, by Lemma \ref{cor_number_independent_extended}, we have that 
$$ \sum\limits_{\substack{u \in V(T) \setminus \{\rho'\} \\ n_u \geq 2}} i_k(T_u) = \sum\limits_{\substack{u \in V(T) \setminus \{\rho'\} \\ n_u \geq 2}} \, \sum\limits_{i=0}^k e_i(T_u^1) \cdot e_{k-i}(T_u^2),$$
where $T_u=(T_u^1, T_u^2)$ denotes the standard decomposition of $T_u$. Using Definition \ref{def_extendedindependent} and Lemma \ref{extended_number_recursion}, we now have
 $$e_i(T_u^1) = \begin{cases}
    1, &\text{if $i=0$} \\
    0, &\parbox[t]{.5\textwidth}{if $i > 0$ and $T_u^1$ contains precisely one leaf.}  \\
    \sum\limits_{j=0}^i e_j(T_u^{11}) \cdot e_{i-j}(T_u^{12}) \\
    \quad + e_{i-1}(T_u^{11}) + e_{i-1}(T_u^{12}), &\parbox[t]{.5\textwidth}{if $i > 0$ and $T_u^1=(T_u^{11},T_u^{12})$ contains at least two leaves,}
    \end{cases} $$
    and 
    $$e_{k-i}(T_u^2) = \begin{cases}
    1 &\text{if $k-i=0$} \\
    0 &\parbox[t]{.45\textwidth}{if $k-i > 0$ and $T_u^2$ contains precisely one leaf.}  \\
    \sum\limits_{j=0}^{k-i} e_j(T_u^{21}) \cdot e_{k-i-j}(T_u^{22})  \\
    \quad + e_{k-i-1}(T_u^{21}) + e_{k-i-1}(T_u^{22}), &\parbox[t]{.45\textwidth}{if $k-i > 0$ and $T_u^2=(T_u^{21},T_u^{22})$ contains at least two leaves.}
    \end{cases} $$
This completes the proof.
\end{proof}

Theorem \ref{thm_counting} suggests that the number of Dollo-$k$ characters a rooted binary phylogenetic tree $T$ induces can be calculated recursively by decomposing $T$ and its subtrees. This is summarized in Algorithm \ref{Alg_Counting}. The subroutine \texttt{extended}$(T',k')$ takes as an input a rooted binary phylogenetic tree $T'$ and an integer $k'$. It returns the number of extended independent node sets of size $k'$ for $T'$ (employing Lemma \ref{extended_number_recursion} in the case that $T'$ contains at least two leaves). The function \texttt{main} then calculates the number of Dollo-$k$ characters for a given rooted binary phylogenetic tree $T$ and a given integer $k \geq 0$ according to Theorem \ref{thm_counting} (for $k=0$, Part (i) of Theorem \ref{thm_counting} is used, and for $k \geq 1$, Part (ii) of Theorem \ref{thm_counting} is used). In order to compute the quantities $i_k(T_u)$ for a subtree $T_u=(T_u^1,T_u^2)$ of $T$ that contains at least two leaves, the subroutine \texttt{extended} is applied to $T_u^1$ and $T_u^2$, respectively. \\

We thus have the following corollary.
\begin{corollary}
Let $T$ be a rooted binary phylogenetic tree and let $k \geq 1$. Then,  $\cD_k(T)$ can be calculated with Algorithm \ref{Alg_Counting}. 
\end{corollary}

\begin{algorithm}[H]\label{Alg_Counting}
\SetKwInput{KwInput}{Input}                % Set the Input
\SetKwInput{KwOutput}{Output}              % set the Output
\DontPrintSemicolon
\KwInput{Rooted binary phylogenetic tree $T$ with $n$ leaves, integer $k \geq 0$}
\KwOutput{Number $\cD_k(T)$ of Dollo-$k$ characters for $T$}
\caption{Compute number of Dollo-$k$ characters}
\vspace*{2mm}
\textbf{Subroutine} \texttt{extended($T',k'$)}\;
$n' \coloneqq $ number of leaves of $T'$\;
\uIf{$k'=0$}{\KwRet 1\;} 
\uElseIf{$k' > 0$ and $n=1$}{\KwRet 0\;}
\Else{
  Decompose $T'$ into its two maximal pending subtrees $T'_a$ and $T'_b$\;
  $e_{k'}(T') \coloneqq 0$\;
  \For{$i=0, \ldots, k'$}{
       $e_{k'}(T') = e_{k'}(T') + \texttt{extended}(T'_a,i) \cdot \texttt{extended}(T'_b,k'-i)$\;
        }
  $e_{k'}(T') = e_{k'}(T') + \texttt{extended}(T'_a,k'-1) + \texttt{extended}(T'_b,k'-1)$\;
  \KwRet $e_{k'}(T')$\;
  }
\vspace*{2mm} 
\textbf{Function} \texttt{main}\;
\uIf{$k=0$}{\KwRet $\mathcal{D}_0(T)=2n$\;}
\Else{
$\cD_k(T) \coloneqq 0$\;
\ForAll{$u \in V(T) \setminus \{\rho'\}$ with $n_u \geq 2$}{
    Decompose $T_u$ into its two maximal pending subtrees $T_u^1$ and $T_u^2$\;
    $i_k(T_u) \coloneqq 0$\;
    \For{$i=0, \ldots, k$}{
        $i_k(T_u) = i_k(T_u) + \texttt{extended}(T_u^1,i) \cdot \texttt{extended}(T_u^2,k-i)$\;
    }
$\cD_k(T) = \cD_k(T) + i_k(T_u)$\;
}
\KwRet $\cD_k(T)$\;}
\end{algorithm}

\begin{remark}
Note that the pseudocode given in Algorithm \ref{Alg_Counting} leads to an exponential run time for calculating the number of Dollo-$k$ characters for a given rooted binary tree $T$ with $n$ leaves. However, by caching the values of $e_k(T)$ and $i_k(T)$, we can derive a $O(k^2 |V(T)|)$ algorithm (in the number of nodes of $T$), i.e. to a polynomial time algorithm (for details see the java code for a modified version of Algorithm \ref{Alg_Counting} in Section \ref{subsec_java} in the Appendix). Also note that when calculating the number of Dollo-$k$ characters for multiple $k$, the values of $e_k(T)$ and $i_k(T)$ can be reused, which further reduces computation times. This polynomial time algorithm has been implemented in the \texttt{DolloAnnotator} app in the Babel package for BEAST 2 \citep{Bouckaert2019}, which is publicly available. 
\end{remark}

Using the \texttt{DolloAnnotator} app the number of Dollo-$k$ characters can quickly be computed, even if the trees are large. As an example, Figure \ref{Fig_c128} in Section \ref{subsec_dollocounts} in the Appendix shows the number of Dollo-$k$ characters for $k=0, \ldots, 128$ for both the fully balanced tree of height 7 (i.e. on 128 leaves) and the caterpillar tree on 128 leaves. It is interesting to note that the `distribution' of the number of  Dollo-$k$ characters for the caterpillar tree is almost symmetric (with the highest numbers occuring for $k=62$ and $k=63$), while this is not the case for the fully balanced tree. In particular, the are no Dollo-$k$ characters with $k > 64 = n/2$ for the latter. This is simply due to the fact that there are no independent node sets of size $k > 2^{h-1}$ for the fully balanced tree $T_h^\mathit{fb}$ of height $h$, as we will show in the following.

\begin{proposition} \label{Prop_MaxK_Balanced}
Let $T_h^\mathit{fb}$ be a fully balanced tree of height $h$ with $h \geq 2$. Then the maximum cardinality of any independent node set for $T$ equals $2^{h-1}$ and such an independent node set of size $2^{h-1}$ always exists. In particular, $i_{2^{h-1}}(T_h^\mathit{fb}) > 0$ and $i_k(T_h^\mathit{fb})=0$  for all $k > 2^{h-1}$.
\end{proposition}

In order to prove Proposition \ref{Prop_MaxK_Balanced}, we require the following lemma, which is basically the corresponding statement for \emph{extended} independent node sets (but note the different minimum value for $h$!).

\begin{lemma}\label{Lemma_MaxEK_Balanced}
Let $T_h^\mathit{fb}$ be a fully balanced tree of height $h$ with $h \geq 1$. Then the maximum cardinality of any extended independent node set for $T$ equals $2^{h-1}$ and such an extended independent node set of size $2^{h-1}$ always exists. In particular, $e_{2^{h-1}}(T_h^\mathit{fb}) > 0$ and $e_k(T_h^\mathit{fb})=0$ for $k > 2^{h-1}$.
\end{lemma}

\begin{proof}
We prove this statement by induction on $h$. For $h=1$, $T_1^\mathit{fb}$ consists of a single cherry, say $[1,2]$, and the maximum size of an extended independent node set for $T_1^\mathit{fb}$ is $1=2^0 = 2^{h-1}$ (there are two extended independent node sets of size 1 for $T_1^\mathit{fb}$, namely $\{1\}$ and $\{2\}$, but there is no extended independent node set of size equal or greater than 2). In particular, $e_1(T_1^\mathit{fb})=2 > 0$ and $e_k(T)=0$ for $k > 1 = 2^{h-1}$. 
Now, assume that the statement holds for all $h' < h$ and consider a fully balanced tree $T_h^\mathit{fb} = (T_{h-1}^\mathit{fb}, T_{h-1}^\mathit{fb})=(T_a,T_b)$ of height $h$. By the inductive hypothesis, the size of a largest extended independent node set for both maximal pending subtrees of $T_h^\mathit{fb}$ equals $2^{h-2}$ (and there are extended independent node sets of this size). By taking the union of such an extended independent node set of size $2^{h-2}$ for $T_a$ and one for $T_b$, we obtain an extended independent node set of size $2 \cdot 2^{h-2} = 2^{h-1}$ for $T_h^\mathit{fb}$. In particular, $e_{2^{h-1}}(T_h^\mathit{fb}) > 0$. Moreover, using Lemma \ref{extended_number_recursion}, it easily follows that $e_k(T_h^\mathit{fb}) = 0$ for $k > 2^{h-1}$. As an example, for $k=2^{h-1}+1$, by Lemma \ref{extended_number_recursion} we have
\begin{align*}
    e_{2^{h-1}+1}(T_h^\mathit{fb}) &= \sum\limits_{i=0}^{2^{h-1}+1} e_i(T_a) \cdot e_{2^{h-1}+1-i}(T_b) + e_{2^{h-1}}(T_a) + e_{2^{h-1}}(T_b)\\
    &= e_0(T_a) \cdot \underbrace{e_{2^{h-1}+1}(T_b)}_{= 0 \text{ by induction}} + \ldots + e_{2^{h-2}}(T_a) \cdot \underbrace{e_{2^{h-2}+1}(T_b)}_{= 0 \text{ by induction}} \\ 
    &\qquad + \underbrace{e_{2^{h-2}+1}(T_a)}_{= 0 \text{ by induction}} \cdot e_{2^{h-2}}(T_b) + \ldots + \underbrace{e_{2^{h-1}+1}(T_a)}_{= 0 \text{ by induction}} \cdot e_0(T_a) \\
    &\qquad + \underbrace{e_{2^{h-1}}(T_a)}_{= 0 \text{ by induction}} + \underbrace{e_{2^{h-1}}(T_b)}_{= 0 \text{ by induction}} \\
    &= 0.
\end{align*}
Analogously, it follows that $e_k(T_h^\mathit{fb}) = 0$ for $k > 2^{h-1}+1$. This completes the proof.
\end{proof}

Using Lemma \ref{cor_number_independent_extended} and Lemma \ref{Lemma_MaxEK_Balanced}, we can now prove Proposition \ref{Prop_MaxK_Balanced}.

\begin{proof}[Proof of Proposition \ref{Prop_MaxK_Balanced}]
Let $T_h^\mathit{fb} = (T_{h-1}^\mathit{fb}, T_{h-1}^\mathit{fb})=(T_a,T_b)$ be a fully balanced tree of height $h$ with $h \geq 2$. Then, by Lemma \ref{Lemma_MaxEK_Balanced}, there exists an extended independent node set of size $2^{h-2}$ for $T_a$, say $\mathcal{E}_{2^{h-2}}^a$, and there exists an extended independent node set of size $2^{h-2}$ for $T_b$, say $\mathcal{E}_{2^{h-2}}^b$. It is now easily verified that $\mathcal{I} \coloneqq \mathcal{E}_{2^{h-2}}^a \cup \mathcal{E}_{2^{h-2}}^b$ yields an independent node set of size $2^{h-1}$ for $T_h^\mathit{fb}$. In particular, $i_{2^{h-1}}(T_h^\mathit{fb}) > 0$. Moreover, using Lemma \ref{cor_number_independent_extended} it follows that $i_k(T_h^\mathit{fb})=0$ for $k > 2^{h-1}$. Exemplarily, for $k=2^{h-1}+1$, we have by Lemma \ref{cor_number_independent_extended}
\begin{align*}
     i_{2^{h-1}+1}(T_h^\mathit{fb}) &= \sum\limits_{i=0}^{2^{h-1}+1} e_i(T_a) \cdot e_{2^{h-1}+1-i}(T_b) \\
    &= e_0(T_a) \cdot \underbrace{e_{2^{h-1}+1}(T_b)}_{= 0 \text{ by Lemma \ref{Lemma_MaxEK_Balanced}}} + \ldots + e_{2^{h-2}}(T_a) \cdot \underbrace{e_{2^{h-2}+1}(T_b)}_{= 0 \text{ by Lemma \ref{Lemma_MaxEK_Balanced}}} \\ 
    &\qquad + \underbrace{e_{2^{h-2}+1}(T_a)}_{= 0 \text{ by Lemma \ref{Lemma_MaxEK_Balanced}}} \cdot e_{2^{h-2}}(T_b) + \ldots + \underbrace{e_{2^{h-1}+1}(T_a)}_{= 0 \text{ by Lemma \ref{Lemma_MaxEK_Balanced}}} \cdot e_0(T_a) \\
    &= 0.
\end{align*}
Analogously, it follows that $i_k(T_h^\mathit{fb}) = 0$ for $k > 2^{h-1}+1$. This completes the proof.
\end{proof}

\subsection{The extremal case of Dollo-$(n-2)$ characters}
While we have seen in Section \ref{section_sackin} that there is no direct correspondence between the balance of a tree and its number of Dollo-$k$ characters, the final aim of this manuscript is to show that there is, however, a relationship between the shape of a tree $T$ and the question whether $T$ induces Dollo-$(n-2)$ characters. 

\begin{proposition} \label{Existence_Dollon-2}
Let $T = (T_a, T_b)$ be a rooted binary phylogenetic tree with $n \geq 3$ leaves and root $\rho$. Let $n_a$ and $n_b$ denote the numbers of leaves of $T_a$ and $T_b$, respectively, where $n_a \geq n_b$. Then, $\cD_k(T)>0$ (i.e. $T$ induces Dollo-$(n-2)$ characters) if and only if $T$ is a semi-caterpillar tree. 
\end{proposition}

\begin{proof}
Let $T$ be a rooted binary phylogenetic tree with root $\rho$ and $n \geq 3$ leaves. In order to analyze the existence of Dollo-$(n-2)$ characters, by Theorem \ref{Thm_sum_ik} we need to analyze the existence of independent node sets of size $n-2$ for $T$.

First, suppose that $T=(T_a,T_b)$ is s semi-caterpillar tree, i.e. both $T_a$ and $T_b$ are rooted caterpillar trees. Let $n_a$ and $n_b$ denote the number of leaves of $T_a$ and $T_b$, respectively, and assume that $n_a \geq n_b$ (and $n_a+n_b=n\geq3$). We now distinguish two cases:

\begin{enumerate}[\rm (i)]
    \item If $n_b=1$, $T_a$ is a rooted caterpillar tree with at least two leaves (as $n_a+n_b \geq 3$) and thus it contains precisely one cherry, say $[x,y]$. Let $b$ denote the single leaf of $T_b$. Then the set $\mathcal{I} \coloneqq X \setminus \{x,b\}$ (and analogously the set $\mathcal{I}' \coloneqq X \setminus \{y,b\}$) is an independent node set of size $n-2$ of $T$ because:
    \begin{itemize}
        \item $|\mathcal{I}|=n-2$.
        \item $\rho, a,b \notin \mathcal{I}$ (i.e. neither the root of $T$ nor its children are contained in $\mathcal{I}$).
        \item $u_s$ is not an ancestor of $u_t$ for all $s \neq t$ (since $\mathcal{I}$ contains only leaves of $T$),
	    \item $u_s$ and $u_t$ are not siblings for all $s \neq t$ (since $\mathcal{I}$ contains only leaves of $T$ and the only leaves that are siblings in $T$ are $x$ and $y$, but $x \notin \mathcal{I}$ (all other leaves in $X$ have internal nodes as siblings)). 
    \end{itemize}
    \item If $n_b \geq 2$, both $T_a$ and $T_b$ are rooted caterpillar trees on at least two leaves and thus they both contain precisely one cherry. Let $[x,y]$ denote the cherry of $T_a$ and let $[v,w]$ denote the cherry of $T_b$. Then the set $\mathcal{I} \coloneqq X \setminus \{x,v\}$ (and analogously the sets $\mathcal{I}' \coloneqq X \setminus \{x,w\}$, $\mathcal{I}'' \coloneqq X \setminus \{y,v\}$ and $\mathcal{I}''' \coloneqq X \setminus \{y,w\}$) is an independent node set of size $n-2$ of $T$, since:
\begin{itemize}
	\item $\vert \mathcal{I} \vert = n-2$.
	\item $\rho, a, b \notin \mathcal{I}$ (i.e. neither the root of $T$ nor its children are contained in $\mathcal{I}$).
	\item $u_s$ is not an ancestor of $u_t$ for all $s \neq t$ (since $\mathcal{I}$ contains only leaves of $T$),
	\item $u_s$ and $u_t$ are not siblings for all $s \neq t$ (since $\mathcal{I}$ contains only leaves of $T$ and the only leaves that are siblings in $T$ are $x$ and $y$, and $v$ and $w$ but $x,v \notin \mathcal{I}$ (all other leaves in $X$ have internal nodes as siblings)). 
	\end{itemize}
    \end{enumerate}
Thus, in both cases there exists an independent node set of size $n-2$ for $T$, and thus by Theorem \ref{Thm_sum_ik}, $\cD_{n-2}(T)>0$. This completes the first part of the proof.

We now show that any rooted binary phylogenetic tree $T=(T_a,T_b)$ with root $\rho$ and $\cD_{n-2}(T) > 0$ has the property that both $T_a$ and $T_b$ are rooted caterpillar trees with $n_a$ and $n_b$ leaves, respectively (where without loss of generality $n_a \geq n_b$).

Therefore, let $\mathcal{I}_{n-2} = \{u_1, \ldots, u_{n-2}\} \in I_{n-2}(T_u)$ be an independent node set of size $n-2$ of some subtree $T_u$ of $T$. 

We first claim that $T_u \equiv T_{\rho} \equiv T$, i.e. no other subtree of $T$ has an independent node set of size $n-2$. This is due to the fact that by Proposition \ref{dollo_0_n-2} we have for any rooted binary phylogenetic tree with $n_u$ leaves and any character $f$ that $d(f,T_u) \leq n_u-2$. As $n_u < n$ for every $T_u \neq T_\rho$ this implies that $T_u \neq T_\rho$ does not have Dollo-$(n-2)$ characters and thus in particular no independent node set of size $n-2$. In particular, $\rho$ is the $\cB$-node.

We now claim that $\mathcal{I}_{n-2}$ can only contain elements of the leaf set $X$ (and no internal node of $T$). In particular, $\mathcal{I}_{n-2} \subset X$. To see this, recall that we can identify the elements of $\mathcal{I}_{n-2}$ with the maximal $\0$-nodes of $T$ (as in the proof of Lemma \ref{Thm_sum_ik}). For the sake of a contradiction, assume that $u_i \in \mathcal{I}_{n-2}$ is an internal node of $T$. Then no node in $V(T_{u_i}) \setminus \{u_i\}$ can be an element of $\mathcal{I}_{n-2}$ (as $u_i$ is an ancestor of all of them). Consider the tree $T'$ obtained from $T$ by deleting all elements in $V(T_{u_i}) \setminus \{u_i\}$. Then, $T'$ is a rooted binary phylogenetic tree with $n'=n-n_{u_i}+1 \leq n-2+1=n-1$ leaves. In particular, by Proposition \ref{dollo_0_n-2} and Corollary \ref{Dollo_count_lemma}, $T'$ has at most $n'-2 \leq n-3$ maximal $\0$-nodes (and $u_i$ is one of them). However, as no node in $V(T_{u_i}) \setminus \{u_i\}$ can be a maximal $\0$-node (as $u_i$ is an ancestor of all of them), this implies that $T$ also has at most $n-3$ maximal $\0$-nodes contradicting the fact that $\mathcal{I}_{n-2}$ is an independent node set of size $n-2$. Thus, all elements contained in $\mathcal{I}_{n-2}$ must be leaves of $T$.

We now claim that both $T_a$ and $T_b$ contain at most one cherry. To see this, first note that as any two elements $u_s$ and $u_t$ of $\mathcal{I}_{n-2}$ cannot be siblings, $T$ can have at most two cherries (as $\mathcal{I}_{n-2}$ contains $n-2$ elements, all of which are leaves, and no pair of these leaves can form a cherry). Now, if $T$ has precisely one cherry (as $n \geq 3$ there must be a cherry), it is clear that both $T_a$ and $T_b$ can have at most one cherry (in fact, as we assume that $n_a \geq n_b$, the cherry is in $T_a$ in this case and $T_b$ consists of a single leaf). In particular, both $T_a$ and $T_b$ are caterpillar trees and therefore, $T$ is a semi-caterpillar. So in this case, there remains nothing to show.
Thus, we now consider the case that $T$ contains precisely two cherries and, for the sake of a contradiction, assume that both of them are in $T_a$. This implies that $T_b$ contains no cherry at all. In particular, $T_b$ consists of a single leaf $b$. As $\rho$ is the $\cB$-node, leaf $b$ must be in state 1, i.e. it cannot be a maximal $\0$-node. This implies that all of the $n-2$ maximal $\0$-nodes are in $T_a$, which is a contradiction (since $n_a=n-1$ and thus by Proposition \ref{dollo_0_n-2} and Corollary \ref{Dollo_count_lemma}, $T_a$ can contain at most $n_a-2=n-3$ maximal $\0$-nodes). This implies that $T_a$ contains at most one cherry. Analogously, we can conclude that $T_b$ contains at most one cherry.  
Thus, we can conclude that both $T_a$ and $T_b$ contain at most one cherry. Therefore, $T_a$ and $T_b$ are both rooted caterpillar trees, and thus $T$ is a semi-caterpillar. This completes the proof.
\end{proof}

A direct consequence of this proposition and its proof is the following corollary:
\begin{corollary}  \label{Number_Dollon-2}
Let $T=(T_a, T_b)$ be a rooted binary phylogenetic tree with $n \geq 3$ leaves. Then, the number of Dollo-$(n-2)$ characters of $T$ is either $0$, $2$ or $4$, i.e. $\mathcal{D}_{n-2}(T) \in \{0,2,4\}$.
\end{corollary}

\begin{proof}
Let $T=(T_a, T_b)$ be a rooted binary phylogenetic tree with $n \geq 3$ leaves. Let $n_a$ and $n_b$ denote the number of leaves of $T_a$ and $T_b$, respectively, where $n_a \geq n_b$.
	\begin{itemize}
	\item If $T=(T_a, T_b)$ does not have the property that both $T_a$ and $T_b$ are rooted caterpillar trees, $\mathcal{D}_{n-2}(T) = 0$ by Proposition \ref{Existence_Dollon-2}.
	\item If $T=(T_a, T_b)$ is such that both $T_a$ and $T_b$ are rooted caterpillar trees and $n_b=1$, $\mathcal{D}_{n-2}(T) = 2$. To see this, let $[x,y]$ denote the cherry of $T_a$ and let $b$ denote the single leaf of $T_b$. As in the proof of Proposition \ref{Existence_Dollon-2} we can conclude that any independent node set of size $n-2$ can only contain elements of the leaf set $X$ and no pair of elements in this set can form a cherry. Moreover, node $b$ cannot be in any independent node set of size $n-2$ of $T$. Thus, the only independent node sets of size $n-2$ of $T$ are $\mathcal{I}^1 = X \setminus\{x,b\}$ and $\mathcal{I}^2 = X \setminus\{y,b\}$ (they indeed \emph{are} independent node sets of size $n-2$ as shown in the first part of the proof of Proposition \ref{Existence_Dollon-2}). Hence $\mathcal{D}_{n-2}(T) = 2$.
	\item If $T=(T_a, T_b)$ is such that both $T_a$ and $T_b$ are rooted caterpillar trees with $n_a, n_b \geq 2$, we have $\mathcal{D}_{n-2}(T) = 4$. To see this, let $[x,y]$ denote the cherry of $T_a$ and let $[u,v]$ denote the cherry of $T_b$. Again, any independent node set of size $n-2$ can only contain elements of the leaf set $X$ and no pair of elements in the independent node set can form a cherry. Thus, there are exactly 4 possible independent node sets of size $n-2$: $\mathcal{I}^1 = X \setminus \{x,u\}$, $\mathcal{I}^2 = X \setminus \{x,v\}$, $\mathcal{I}^3 = X \setminus \{y,u\}$ and $\mathcal{I}^4 =  X \setminus \{y,v\}$ (they indeed \emph{are} independent node sets of size $n-2$ as shown in the first part of the proof of Proposition \ref{Existence_Dollon-2}). 
	\end{itemize}
	This completes the proof.
\end{proof}

Thus, any tree on $n \geq 3$ leaves induces either $0$, $2$ (Figure \ref{Fig_Caterpillar1}) or $4$ (Figure \ref{Fig_Caterpillar2}) Dollo-$(n-2)$ characters and this depends on whether $T$ is a caterpillar tree, a semi-caterpillar tree, or none of them.

\begin{figure}[h!]
	\centering
	\includegraphics[scale=0.3]{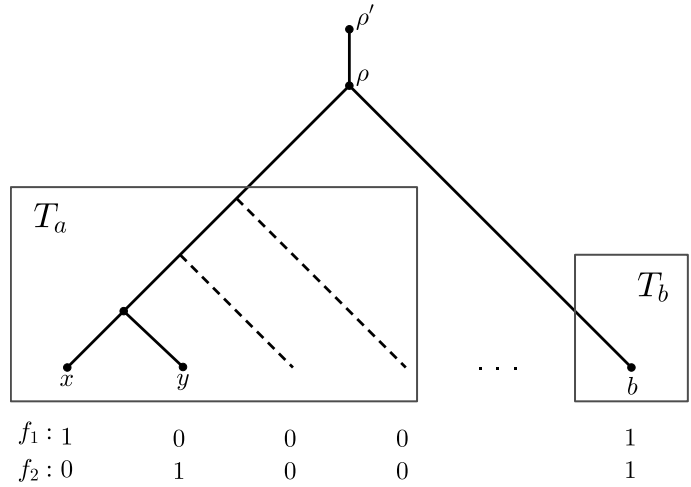}
	\caption{Rooted binary phylogenetic tree $T=(T_a, T_b)$ such that both $T_a$ and $T_b$ are rooted caterpillar trees, where $n_b=1$ and $n_a \geq 2$. The two characters depicted at the leaves are the only Dollo-$(n-2)$ characters of $T$. }
	\label{Fig_Caterpillar1}
\end{figure}

\begin{figure}[h!]
	\centering
	\includegraphics[scale=0.125]{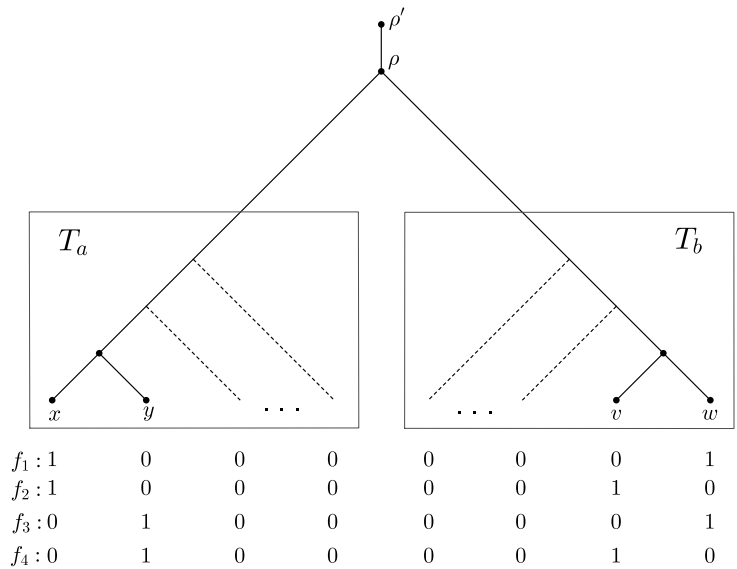}
	\caption{Rooted binary phylogenetic tree $T=(T_a, T_b)$ such that both $T_a$ and $T_b$ are rooted caterpillar trees, where $n_a, n_b \geq 2$. The four characters depicted at the leaves are the only Dollo-$(n-2)$ characters of $T$.}
	\label{Fig_Caterpillar2}
\end{figure}

\section{Discussion}
So-called persistent characters have recently attracted great attention in the literature, both from an algorithmic point of view  \citep{Bonizzoni2012,Bonizzoni2014,Bonizzoni2016,Bonizzoni2017} as well as from a combinatorial perspective \citep{Wicke2018}.

Here, we have taken a combinatorial perspective on the more general notion of Dollo-$k$ characters that generalize persistent characters (as persistent characters are simply the union of Dollo-0 and Dollo-1 characters) and have thoroughly analyzed their properties. 

First of all, we have introduced an Algorithm (Algorithm \ref{Alg_Dollo_1tree}) that can be used to calculate the Dollo-score as well as a Dollo-$k$ labeling for a binary character $f$ on a rooted binary phylogenetic tree $T$ in linear time by considering the 1-tree, i.e. the minimum subtree of $T$ that connects all leaves assigned state 1 by $f$. 

While this Algorithm and the 1-tree can also be used to characterize persistent characters, thereby complementing a  characterization of persistent characters based on the Fitch algorithm obtained in \citep{Wicke2018}, we have then highlighted that there are striking differences between persistent characters and general Dollo-$k$ characters. 
First, we have shown that while there is a close relationship between Fitch parsimony and persistent characters, this is not the case for general Dollo-$k$ characters. More precisely, the absolute difference between the parsimony score and the Dollo score of a character $f$ can be made arbitrarily large. 
Second, while there is a direct connection between the number of persistent characters and the Sackin index of a rooted binary phylogenetic tree $T$, we showed that this correspondence does not generalize to Dollo-$k$ characters.

In fact, counting the number of Dollo-$k$ characters turned out to be much more involved than counting the number of persistent characters and we have devoted the last part of this manuscript to establishing a recursive approach (cf. Theorem \ref{thm_counting} for this task, resulting in a polynomial-time algorithm. Both this algorithm as well as the algorithm for computing the Dollo-score and Dollo-labeling for a binary character $f$ have been added to the Babel package of BEAST 2 \citep{Bouckaert2019} (in form of the \texttt{DolloAnnotator} app) and are publicly available. Nevertheless, it would definitely be of interest to further study the number of Dollo-$k$ characters for a given tree $T$ and analyze whether it is possible to find an explicit formula for the quantities $\cD_k(T)$, e.g. based on the shape of $T$. We leave this as an open problem for future research.

\section*{Authors' contributions}
MF and KW contributed most (but not all) of the mathematical contents, while RB contributed the \texttt{DolloAnnotator} implementation.

\section*{Acknowledgments}
Mareike Fischer thanks the joint research project \textit{\textbf{DIG-IT!}} supported by the European Social Fund (ESF), reference: ESF/14-BM-A55-0017/19, and the Ministry of Education, Science and Culture of Mecklenburg-Vorpommern, Germany.
Moreover, Kristina Wicke thanks the German Academic Scholarship Foundation for a doctoral scholarship, under which parts of this work were conducted. All authors thank an anonymous reviewer for valuable suggestions on an earlier version of this manuscript.

\bibliographystyle{abbrvnat}
\bibliography{references_dollo}

\section{Appendix} \label{sec_Appendix}

\subsection{Java code for Algorithm \ref{Alg_Counting}} \label{subsec_java}
Below we present an efficient Java implementation of Algorithm \ref{Alg_Counting}. In contrast, to the pseudocode given in the main part of this paper, it caches the values for $i_k(T)$ and $e_k(T)$. Note that the \verb|ikt| and \verb|ekt| arrays are of dimension $|V(T)| \cdot (k+1)$. Since calculating each entry takes $k$ calculations, the complexity of the algorithm is in $O(k^2 |V(T)|)$. This java implementation (with an additional BigDecimal version) is used in the \texttt{DolloAnnotator} app in the Babel package for BEAST 2 \citep{Bouckaert2019} that is publicly available.

\begin{lstlisting}[language=Java,frame=single,caption=Efficient Java implementation of Algorithm \ref{Alg_Counting}., label=alg2java]
/* 
* Count the number of possible Dollo-k characters (that is, Dollo-assignments requiring k death events) on a given tree.
* 
* Implements Algorithm 2 of the present manuscript with ik and ek cached.
*
* @param tree
* @param k
* @param ik: independent node set of size k for subtree under node. Must be initialised as -1 at first call.
* @param ek: extended independent node set of size k for subtree under node. Must be initialised as -1 at first call.
* @return number of possible Dollo-k characters
* throws MathException when encountering underflow
*/
public long dolloKCount(Tree tree, int k, long [][] ik, long [][] ek) throws MathException {
		if (k == 0) {
			return tree.getNodeCount() + 1;
		}
		long count = 0;
		for (int j = tree.getLeafNodeCount(); j < tree.getNodeCount(); j++) {
			if (verboseInput.get()) System.err.print(j % 10 == 0 ? '|' : '.');
			Node node = tree.getNode(j);
			Node left = node.getLeft();
			Node right = node.getRight();
			long ikt = ik[node.getNr()][k];
			if (ikt < 0) {
				ikt = 0;
				for (int i = 0; i <= k; i++) {
					ikt += extended(left, i, ek) * extended(right, k-i, ek);
				}					
				ik[node.getNr()][k] = ikt;
			}
			count += ikt;
		}
    	if (count < 0) {
    		throw new MathException("Underflow encountered! Count > " + Long.MAX_VALUE + " ");
    	}
		return count;
	}

	/** extended method from Algorithm 2, but with ek cache for extended independent node set of size k for subtree under node **/
	private long extended(Node node, int k, long [][] ek) {
		if (k == 0) {
			return 1;
		}
		int n = node.getLeafNodeCount();
		if (k > 0 && n == 1) {
			return 0;
		}
		Node left = node.getLeft();
		Node right = node.getRight();
		long ekt = ek[node.getNr()][k];
		if (ekt < 0) {
			ekt = 0;
			for (int i = 0; i <= k; i++) {
				ekt += extended(left, i, ek) * extended(right, k-i, ek);
			}
			ekt = ekt + extended(left, k-1, ek) + extended(right, k-1, ek);
			ek[node.getNr()][k] = ekt;
		}
		return ekt;
	}
\end{lstlisting}

\subsection{Number of Dollo-$k$ characters for the fully balanced tree of height seven and the caterpillar tree on 128 leaves} \label{subsec_dollocounts}
\begin{figure}
    \centering
    \includegraphics[scale=0.6]{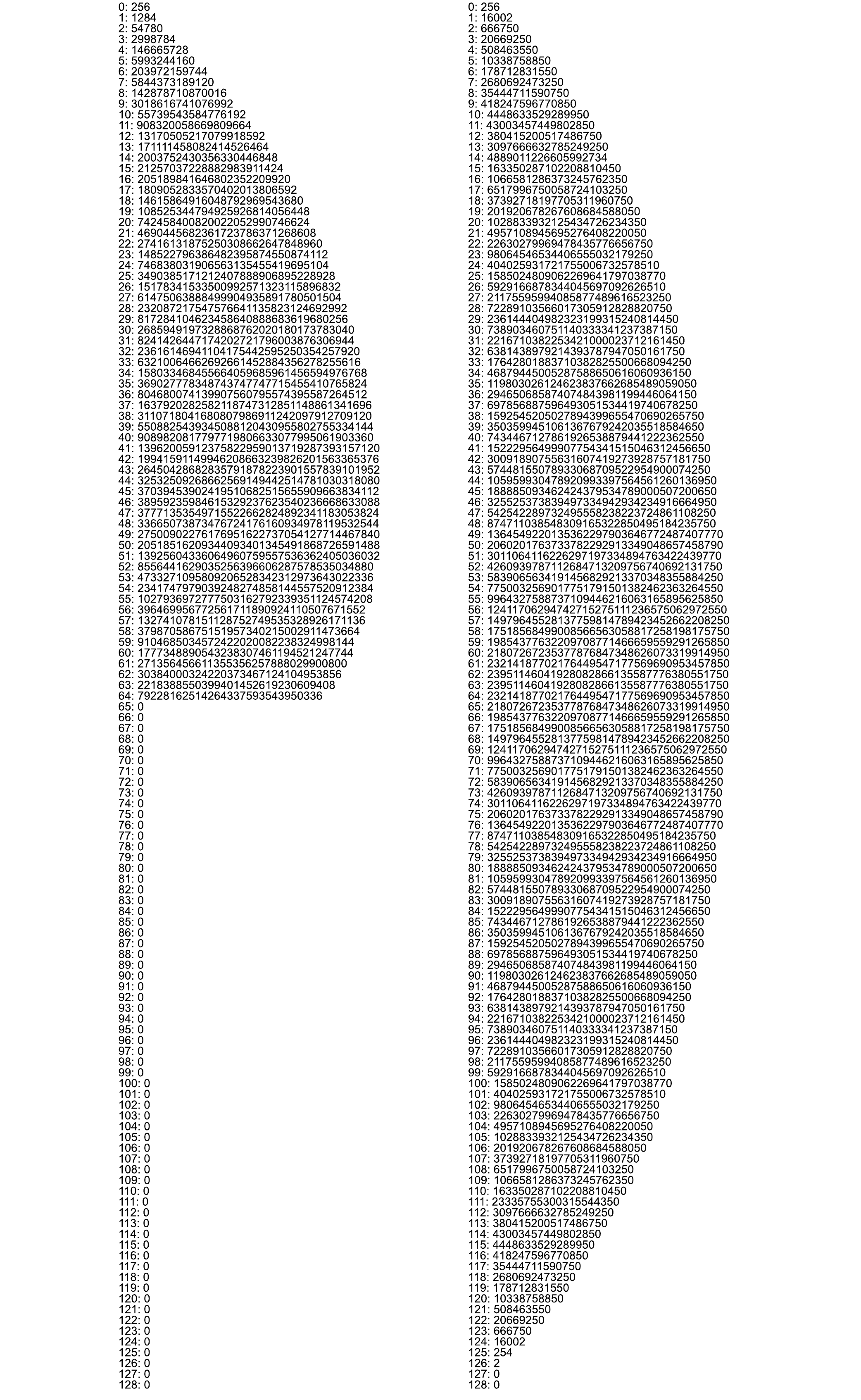}
    \caption{Number of Dollo-$k$ characters for $k=0, \ldots, 128$ for the fully balanced tree of height 7 (left) and the caterpillar tree on 128 leaves (right). The obvious cut-off in the left figure, which is due to the 0's for all values larger than 64, is explained by Proposition \ref{Prop_MaxK_Balanced}.}
    \label{Fig_c128}
\end{figure}

\end{document}